\documentclass[a4paper,10pt]{article}
\usepackage[utf8]{inputenc}
\usepackage[margin=1in]{geometry}
\usepackage{amssymb}
\usepackage{amsmath}
\usepackage{amsthm}
\usepackage{mathtools}
\usepackage{hyperref}
\usepackage{etoolbox}
\usepackage{comment}

\theoremstyle{theorem}{
\newtheorem{thrm}{Theorem}[section]
\newtheorem*{thrm*}{Theorem}
\newtheorem{cor}[thrm]{Corollary}
\newtheorem{conj}[thrm]{Conjecture}
\newtheorem{lma}[thrm]{Lemma}
\newtheorem{propn}[thrm]{Proposition}
\newtheorem{statement}[thrm]{Statement}

\newtheorem*{thrmREP}{Theorem \ref{thrm:PartialAOwDerivs}}
}

\theoremstyle{definition}{
\newtheorem{defn}[thrm]{Definition}
\newtheorem*{defn*}{Definition}
\newtheorem{eg}[thrm]{Example}
\newtheorem*{eg*}{Example}

\newtheorem{notn}[thrm]{Notation}
}

\newcommand{\Imm}{\operatorname{Im}}
\newcommand{\gl}{\operatorname{GL}_2^+(\mathbb{Q})}
\newcommand{\glr}{\operatorname{GL}_2^+(\mathbb{R})}

\newcommand{\slr}{\operatorname{SL}_2(\mathbb{R})}
\newcommand{\slz}{\operatorname{SL}_2(\mathbb{Z})}
\newcommand{\uh}{\mathcal{H}}
\newcommand{\suq}{\subseteq}
\newcommand{\Raxp}{\mathbb{R}_\text{an,exp}}
\newcommand{\JClose}[1]{\left<\left<#1\right>\right>}

\newcommand{\Sp}{\operatorname{Sp}}

\title{A Modular Andr\'e-Oort Statement with Derivatives}
\author{Haden Spence}

\begin{document}

\maketitle

\begin{abstract}
 In unpublished notes, Pila discussed some theory surrounding the modular function $j$ and its derivatives.  A focal point of these notes was the statement of two conjectures regarding $j$, $j'$ and $j''$: a Zilber-Pink type statement incorporating $j$, $j'$ and $j''$, which was an extension of an apparently weaker conjecture of Andr\'e-Oort type.  In this paper, I first cover some background regarding $j$, $j'$ and $j''$, mostly covering the work already done by Pila.  Then I use a seemingly novel adaptation of the o-minimal Pila-Zannier strategy to prove a weakened version of Pila's ``Modular Andr\'e-Oort with Derivatives'' conjecture.  Under the assumption of a Schanuel-type conjecture, the central theorem of the paper implies Pila's conjecture in full generality, as well as a more precise statement on the same lines.
\end{abstract}

\section{Introduction}
The modular Andr\'e-Oort Conjecture is a statement about the arithmetic and algebraic properties of the classical modular function 
\[j:\uh\to\mathbb{C}.\]
The statement is the following.
\begin{thrm}[Pila, Modular Andr\'e-Oort]\label{thrm:ModularAO}
 Let $V\suq\mathbb{C}^n$ be an algebraic variety.  Then $V$ contains only finitely many maximal special subvarieties.
\end{thrm}
While various partial \cite{Andre1998} and conditional \cite{Klingler2014} results were known, this was first proven unconditionally and in full generality by Pila, in his 2011 paper \cite{Pila2011}, using what is now a fairly standard strategy employing ideas of o-minimality and point-counting.  The connection with the $j$ function is obscured behind the definition of ``special subvariety'', which goes as follows.  It is well known that there are modular polynomials $\Phi_N\in\mathbb{Z}[X,Y]$ with the property that 
\[\Phi_N(j(g\tau),j(\tau))=0\]
whenever $g\in\gl$ is a primitive integer matrix of determinant $N$.   So, although $j$ is a transcendental function, it is very well-behaved under the action of $\gl$.  A fairly direct consequence of the existence of the polynomials $\Phi_N$ is the (also well known) fact that $j(\tau)$ is an algebraic integer whenever $\tau\in\uh$ is quadratic.

Loosely, a special subvariety of $\mathbb{C}^n$ is a variety induced by these relations.  To be more precise, we have the following definition.
\begin{defn}\label{defn:geodesics}
 Let $n\in\mathbb{N}$.

 Let $S_0\cup S_1\cup\dots\cup S_k$ be a partition of $\{1,\dots,n\}$, where $k\geq 0$ and $S_i\ne\emptyset$ for $i>0$.
 For each $s\in S_0$, choose any point $q_s\in\uh$.  For each $i>0$, let $s_i$ be the least element of $S_i$ and for each $s_i\ne s\in S_i$ choose a geodesic matrix $g_{i,s}\in\gl$.  A \emph{weakly $\uh$-special} subvariety of $\uh^n$ is a set of the form
 \[\{(\tau_1,\dots,\tau_n)\in\uh^n:\tau_s=q_s\text{ for }s\in S_0, \tau_{s}=g_{i,s}\tau_{s_i} \text{ for }s\in S_i, s\ne s_i,i=1,\dots,k\},\]
 for some given data $S_i$, $q_s$, $g_{i,s}$.
 
 A weakly $\uh$-special subvariety is \emph{$\uh$-special} if the constant factors $q_s$ are imaginary quadratic numbers for all $s\in S_0$. 
\end{defn}
One can define special (henceforth to be known as $j$-special) subvarieties in a similar way, as varieties in $\mathbb{C}^n$ cut out by the modular polynomials $\Phi_N$.  For our purposes, however, there is a simpler definition that suits better.  By abuse of notation, we will write $j$ for the function
\[(\tau_1,\dots,\tau_n)\mapsto (j(\tau_1),\dots,j(\tau_n)).\]
A $j$-special variety in $\mathbb{C}^n$ is then the image under $j$ of an $\uh$-special variety; similarly a weakly $j$-special variety is the image under $j$ of a weakly $\uh$-special variety.  \\

In this paper, I will be investigating what happens when we consider not only $j$ and its corresponding special varieties, but also the derivatives of $j$.  In this setup, the situation is rather more complicated.  To begin with, recall that $j$ satisfies a certain 3rd order differential equation.  Hence it is enough to consider only $j$ and its first two derivatives.  With this in mind, let us define a function which will be our central object of study for the paper:
\[J=(j,j',j''):\uh\to\mathbb{C}^3,\]
\[J(\tau)=(j(\tau),j'(\tau),j''(\tau)).\]
Again we will abuse notation and use $J$ also to refer to the obvious map from $\uh^n$ to $\mathbb{C}^{3n}$ defined as $J$ on each coordinate.

Much of the setup here is due to Pila; in unpublished notes, he compiled various properties of $J$ and gave definitions of what a $J$-special variety should be.  Pila also made a conjecture analogous to \ref{thrm:ModularAO}, with respect to $J$.  For the rest of this section, I will give some of the setup covered by Pila in his notes.  Towards the end of the section I state Pila's conjecture and a weakened version of the conjecture which is the central theorem of the paper.

\subsection{Properties of $J$}
The functions $j'$ and $j''$ are not fully modular functions.  Instead, they satisfy
\[j'(\gamma\tau)=(c\tau+d)^2j'(\tau)\]
and
\[j''(\gamma\tau)=(c\tau+d)^4j''(\tau)+2c(c\tau+d)^3j'(\tau),\]
where $\gamma=\begin{pmatrix}a&b\\c&d\end{pmatrix}\in\slz$.  This says that $j'$ is a meromorphic modular form of weight 2, while $j''$ is a so-called \emph{quasimodular} form of weight 4 and depth 1.  The behaviour of $j'$ and $j''$ at quadratic points is also worse than that of $j$; while $j(\tau)$ is algebraic at quadratic $\tau$, $j'(\tau)$ and $j''(\tau)$ almost never are.  In fact, such points are always transcendental unless $\tau$ is in the $\slz$-orbit of $i$ or $\rho=e^{\frac{2\pi i}{3}}$ (see \cite{Diaz2000}).  However, thanks to a result of Masser, $j''(\tau)$ is always algebraic over $j'(\tau)$.  To be precise, we have a rational function $p_c$, in 3 variables (with $c$ a positive real), defined as follows:
\[p_c(W,X,Z)=\dfrac{1}{6}\dfrac{Z^2(X-7W+6912)}{W(W-1728)}+\dfrac{iZ}{c}.\]
A brief calculation shows that
\[p_{\Imm\tau}(j(\tau),\chi^*(\tau),j'(\tau))=j''(\tau)\]
for all $\tau$.  Here and throughout, $\chi^*$ is an almost holomorphic modular (AHM) function defined by
\[\chi^*=1728\cdot\dfrac{E_2^*E_4E_6}{E_4^3-E_6^2},\]
where the $E_k$ are standard Eisenstein series and $E_2^*$ is the weight 2 almost holomorphic modular form defined by
\[E_2^*(\tau)=E_2(\tau)-\dfrac{3}{\pi\Imm\tau}.\]
We can decompose $\chi^*$ as 
\[\chi^*=\chi-\dfrac{3}{\pi y}\cdot f,\]
where $\chi$ and $f$ are, here and throughout, holomorphic functions defined by
\[\chi=1728\cdot\dfrac{E_2E_4E_6}{E_4^3-E_6^2},\qquad f=1728\cdot\dfrac{E_4E_6}{E_4^3-E_6^2}.\]
The function $\chi^*$ and related AHM functions have been studied in several places, by Masser \cite{Masser1975}, Mertens and Rolen \cite{Mertens2015} and Zagier \cite{Zagier2008} among others.  I have studied $\chi^*$ in the context of an Andr\'e-Oort result \cite{Spence2016}; we will be making use of some results from that paper occasionally.

For the present, the most relevant fact is that $\chi^*(\tau)$ is algebraic whenever $\tau$ is quadratic\footnote{This was essentially proven by Masser in the Appendix of \cite{Masser1975}, which seems to have been the first investigation of the algebraic properties of AHM functions.}.  As a consequence, we see that $p_{\Imm\tau}$ has coefficients in $\overline{\mathbb{Q}}$ whenever $\tau$ is quadratic.  Hence, as claimed, $\text{tr.deg.}_\mathbb{Q}(j(\tau),j'(\tau),j''(\tau))=1$ for quadratic $\tau\not\in \slz\cdot\{i,\rho\}$.

By differentiating the modular polynomials, we see that $J$ also has nice behaviour with respect to $g\in\gl$.  We will briefly work out the details of this in the case of $j'$.  Let $g\in\gl$ be a primitive integer matrix of determinant $N$.  Since
\[\Phi_N(j(\tau),j(g\tau))=0,\]
we see that
\[\partial_X(\Phi_N)(j(\tau),j(g\tau))j'(\tau)+\partial_Y(\Phi_N)(j(\tau),j(g\tau))j'(g\tau)\dfrac{d(g\tau)}{d\tau}=0\]
\[\implies j'(g\tau)=-\lambda_N(j(\tau),j(g\tau))j'(\tau)m_g(\tau),\]
where \[\lambda_N=\dfrac{\partial_X(\Phi_N)}{\partial_Y(\Phi_N)}\]
and 
\[m_g(\tau)=\dfrac{(c\tau+d)^2}{\det g}=\left(\dfrac{d(g\tau)}{d\tau}\right)^{-1}.\]
For $j''$, we have a similar relation, though a bit more complex.  In both cases, the nature of the relation differs depending on whether $g$ is upper triangular.  If $g$ fails to be upper triangular (ie. $c\ne 0$) then the relation between the functions $j'(g\tau)$ and $j'(\tau)$ (respectively for $j''$) only exists over $\mathbb{C}(\tau)$.  Otherwise the relationship is over $\mathbb{C}$.  Hence
\[\text{tr.deg.}_\mathbb{C}(J(\tau),J(g\tau))=3\]
(considering $\tau$ as a variable in $\uh$) if $g$ is upper triangular and
\[\text{tr.deg.}_\mathbb{C}(J(\tau),J(g\tau))=4\]
otherwise.

Since the behaviour of $J$ is affected by whether or not a matrix is upper triangular, it makes sense to make the following definition.
\begin{defn}
 A weakly $\uh$-special variety $G$ is called a geodesic upper-triangular (GUT) variety if all of the $g_{i,s}\in\gl$ arising in its definition (see \ref{defn:geodesics}) are upper triangular matrices.
\end{defn}
We will be making use of GUT varieties quite often, later on in this paper.

\subsection{Special sets for $J$}
In the classical case, the special sets were just the images, under $j$, of $\uh$-special sets.  An important feature is that they are bi-algebraic; for $\uh$-special sets $G$, both $G$ and $j(G)$ are algebraic sets defined over $\overline{\mathbb{Q}}$.  In the $J$ case, as noted earlier, $J(G)$ is not necessarily an algebraic set, let alone defined over $\overline{\mathbb{Q}}$.  The solution to this is fairly simple; just take Zariski closures over $\overline{\mathbb{Q}}$.  We are still following Pila, and will use the following notation, also of his design.
\begin{notn}
 For any subset $S\suq\uh^n$, define $\JClose{S}$ to be the $\overline{\mathbb{Q}}$-Zariski closure of $J(S)$.  That is, $\JClose{S}$ is the smallest algebraic variety, defined over $\overline{\mathbb{Q}}$, which contains $J(S)$.
\end{notn}
\begin{defn}
 A $J$-special subvariety of $\mathbb{C}^{3n}$ is an irreducible component of any set of the form $\JClose{G}$, where $G$ is an $\uh$-special set.
\end{defn}
With this definition made, we might conjecture a direct analogue of \ref{thrm:ModularAO}, with $j$-special varieties replaced by $J$-special ones.  This fails for a fairly obvious reason.  Consider the variety $V\suq\mathbb{C}^3$ defined just by
\[X_1=j(\tau),\]
for some fixed quadratic $\tau\in\uh\setminus \slz\cdot\{i,\rho\}$.  Then by modularity of $j$, $V$ contains all the points $J(\gamma\tau)$, $\gamma\in\slz$.  In particular, taking Zariski closures, we have
\[\JClose{\gamma\tau}\suq V\]
for all $\gamma\in\slz$.  Since
\[\JClose{\gamma\tau}=\{(j(\tau),w,p_{\Imm\gamma\tau}(j(\tau),\chi^*(\tau),w)):w\in\mathbb{C}\},\]
one sees that the various $\JClose{\gamma\tau}$ are in fact distinct.  So $V$ contains infinitely many distinct $J$-special sets.  They are maximal since $\JClose{\uh}=\mathbb{C}^3\not\suq V$.  

With this in mind, we need a version of \ref{thrm:ModularAO} which takes into account the action of $\slz$.  The aforementioned conjecture of Pila is one such version.  To state it, we will need a definition.

\begin{defn}
Let $\mathcal{S}$ be a collection of subsets of $\uh^n$.  We say $\mathcal{S}$ is $\slz$-finite if there is some finite subcollection $\mathcal{T}\suq \mathcal{S}$, such that every $S\in\mathcal{S}$ takes the form
\[S=\gamma\cdot T,\]
for some $\gamma\in\slz^n$, $T\in\mathcal{T}$.  Otherwise, $\mathcal{S}$ is $\slz$-infinite.  
 
 Abusing notation slightly, we also use this terminology to apply to collections of points, equating points $\tau$ with singleton sets $\{\tau\}$ in the obvious way.
\end{defn}

Now we can state Pila's conjecture.
\begin{conj}[Pila, ``Modular Andr\'e-Oort with Derivatives'']\label{conj:PartialAOwDerivs}\hfill\\
 Let $V\suq\mathbb{C}^{3n}$ be a proper algebraic variety defined over $\overline{\mathbb{Q}}$.  There exists an $\slz$-finite collection $\sigma(V)$, consisting of \emph{proper} $\uh$-special varieties of $\uh^n$, with the following property.  Every $J$-special subvariety of $V$ is contained in $\JClose{G}$ for some $G\in\sigma(V)$.
\end{conj}
We will be approaching this conjecture using a variant of the Pila-Zannier strategy and o-minimality.  With a seemingly novel adaptation of the usual strategy, we are able to make some good progress.  The o-minimal methods are sufficient to give us good control over quadratic points of the form 
\[(g_1\sigma,\dots,g_{n}\sigma).\]
In this case, the methods yield a bound on the size of the determinants of the $g_i\in\gl$ and on the discriminant of $\sigma$.  This is a good step towards Conjecture \ref{conj:PartialAOwDerivs}.

There is significant difficulty, however, in dealing with points having a more complex $\gl$-structure.  The difficulty lies in the possibility that, for two quadratic points $\tau,\sigma\in\uh$, it might happen that $j'(\tau)$ and $j'(\sigma)$ are algebraically dependent even when $\tau$ and $\sigma$ lie in distinct $\gl$-orbits.  As we will discuss shortly, we do not expect this to happen, but it does not seem to be possible to exclude the possibility using o-minimal methods.  So we need the following definition.

\begin{defn}
 Let $\tau=(\tau_1,\dots,\tau_n)\in\uh^n$ be a quadratic point.  Then $\tau$ may be written as 
 \[\tau=(\sigma_1,g_{1,1}\sigma_1,\dots,g_{1,r_1}\sigma_1,\dots,\sigma_k,g_{k,1}\sigma_k,\dots,g_{k,r_k}\sigma_k),\]
 with $g_{i,j}\in\gl$ and the $\sigma_i$ lying in distinct $\gl$-orbits.  We say $\tau$ is $j'$-generic if the numbers $j'(\sigma_1),\dots,j'(\sigma_k)$ are algebraically independent over $\overline{\mathbb{Q}}$.
\end{defn}
Now we can state the central theorem of the paper.

\begin{thrm}\label{thrm:PartialAOwDerivs}
 Let $V\suq\mathbb{C}^{3n}$ be a proper algebraic variety defined over $\overline{\mathbb{Q}}$.  There exists an $\slz$-finite collection $\sigma(V)$, consisting of \emph{proper} $\uh$-special varieties of $\uh^n$, with the following property.  Every $j'$-generic $\uh$-special point in $J^{-1}(V)$ is contained in some $G\in\sigma(V)$.
\end{thrm}
Note in particular that any special point of the form $(g_1\sigma_1,\dots,g_{n}\sigma_1)$, with $g_i\in\gl$, is automatically $j'$-generic, unless $g_1\sigma_1$ is in the $\slz$-orbit of $i$ or $\rho$.  Hence we have the following easy corollaries.  
\begin{cor}
 Let $V\suq\mathbb{C}^{3n}$ be a proper algebraic variety defined over $\overline{\mathbb{Q}}$.  There exists an $\slz$-finite collection $\sigma(V)$, consisting of \emph{proper} $\uh$-special varieties of $\uh^n$, with the following property.  Every $\uh$-special point of the form $(g_1\sigma,\dots,g_{n}\sigma)\in J^{-1}(V)$, $g_i\in\gl$, is contained in some $G\in\sigma(V)$.
 \begin{proof}
  This follows directly from \ref{thrm:PartialAOwDerivs}, except for the possible existence of points with a coordinate lying in $\slz\cdot\{i,\rho\}$.  But any such points are automatically contained within $\slz$-finitely many proper $\uh$-special varieties.
 \end{proof}
\end{cor}
\begin{cor}
 Let $V\suq\mathbb{C}^3$ be a proper algebraic variety defined over $\overline{\mathbb{Q}}$.  Then $J^{-1}(V)$ contains only $\slz$-finitely many $\uh$-special points.
\end{cor}

This last, of course, is simply Conjecture \ref{conj:PartialAOwDerivs} for $n=1$.  To get \ref{conj:PartialAOwDerivs} in full generality, we would need something like the following.\begin{conj}\label{conj:AlgIndep}
 Let $\tau_1,\dots,\tau_n\in\uh$ be quadratic points, lying in distinct $\gl$-orbits, none of which lies in the $\slz$-orbit of $i$ or $\rho$.  Then $j'(\tau_1),\dots,j'(\tau_n)$ are algebraically independent over $\overline{\mathbb{Q}}$.
\end{conj}
Otherwise put: ``All quadratic points in $\uh^n$ are $j'$-generic, except those with a coordinate in the $\slz$-orbit of $i$ or $\rho$.''  In this form, it is easy to see that the conjecture, together with Theorem \ref{thrm:PartialAOwDerivs}, implies \ref{conj:PartialAOwDerivs}.

\begin{thrm}
 Assume Conjecture \ref{conj:AlgIndep}.  Then Conjecture \ref{conj:PartialAOwDerivs} holds.
 \begin{proof}
  Immediate from Theorem \ref{thrm:PartialAOwDerivs}.
 \end{proof}

\end{thrm}

Should we believe Conjecture \ref{conj:AlgIndep}?  It is quite strong, having a similar flavour to existing modular Schanuel statements; but it does fit into the existing body of conjectures.  See, for instance, Bertolin's elliptico-toric conjecture (CET) \cite{Bertolin2002204}, from which Conjecture \ref{conj:AlgIndep} follows immediately.  In turn, CET is a special case of the Grothendieck-Andr\'e period conjecture.  So \ref{conj:AlgIndep} fits well with what we might expect.  It is also much stronger than we need; to get Conjecture \ref{conj:PartialAOwDerivs}, various weaker transcendence statements would suffice.  The weaker statements are much less clean to state and fit less obviously into the existing literature, so we stick with Conjecture \ref{conj:AlgIndep} for this paper.\\

The paper is broken down as follows.  Section 2 is dedicated towards proving some Ax-Lindemann type results.  In this section lies the primary novelty of the paper; the Ax-Lindemann results we prove here are necessarily of a very new and unusual shape, in order to account for some problems that arise with point-counting.  This section is completely independent from Conjecture \ref{conj:AlgIndep} and does not involve considerations of $j'$-genericity.  It is in section 3, where we discuss the point-counting aspects, that the issue of $j'$-genericity arises.  Section 4 brings everything together to conclude the proof of \ref{thrm:PartialAOwDerivs}.  The penultimate section of the paper, section 5, is dedicated towards proving a more precise version of Conjecture \ref{conj:PartialAOwDerivs}, under the assumption of Conjecture \ref{conj:AlgIndep}.  In the final section we apply this more precise result, together with a slight adaptation of work of Scanlon \cite{Scanlon2004}, to produce versions of our results which are uniform in algebraic families.\\
\\
\textbf{Acknowledgements.} I would like to take the opportunity to thank Jonathan Pila for his invaluable guidance and supervision, as well as for many excellent suggestions on the subject of this paper.  I would also like to thank Sebastian Eterovi\'c for several useful conversations on these topics.  Last but certainly not least, thanks go to my father Derek for our many discussions and his keen proof-reading eyes!

\section{Ax-Lindemann}
\subsection{Technicalities}
We begin with a few crucial definitions.

\begin{defn}
 A subset of $\uh^n$ is a \emph{linear variety} if, up to reordering of coordinates, it takes the form
 \begin{equation}\label{eqn:BasicLinearSub}G=\{(g_{1,1}\tau_1, \dots g_{1,r_1}\tau_1,\dots,g_{k,1}\tau_k,\dots,g_{k,r_k}\tau_k,t_{0,1},\dots,t_{0,r_0}):\tau_i\in\uh\},\end{equation}
 where $t_{0,j}\in\uh$ and $g_{i,j}\in\slr$.
 A linear variety is called \emph{basic} if $r_0=0$.  So a basic linear variety is defined by finitely many matrices $g_{i,j}\in\slr$, together with a permutation of the coordinates.  We will often suppress the permutation of coordinates, assuming for simplicity that the variety takes exactly the form in (\ref{eqn:BasicLinearSub}).
 
 Any linear variety $G\suq\uh^n$ has an underlying basic variety attached to it, namely the $B\suq \uh^k$ attained by ignoring the constant coordinates $t_{0,j}$.  We say that $G$ is a \emph{translate} of $B$ (by the $t_{0,j}$), or that $B$ is the basic variety underlying $G$.
\end{defn}

As we will see in later sections, the usual counting methods don't work out perfectly when applied to the derivatives problem.  One would like to be able to count linear subvarieties contained in $J^{-1}(V)$, but the methods are slightly too coarse for this.  One can, however, count those varieties which are, in some sense, \emph{approximately} in $J^{-1}(V)$.  This motivates the following.

\begin{defn}
 Let $V\suq\mathbb{C}^{3n}$ be an algebraic variety and $B\suq\uh^k$ a basic linear variety, given by data $g_{i,j}$ as in (\ref{eqn:BasicLinearSub}).  For each $g_{i,j}$, take a complex number $z_{i,j}$ and a real $c_{i,j}>0$.  Also take $n-k$ triples of complex numbers $(w_i,x_i,y_i)$.  

 We say $B$ is \emph{adjacent to} $V$ \emph{via} $z_{i,j},c_{i,j},w_i,x_i,z_i$ if for all $\tau_i$, we have, up to permutation of coordinates,
 \[\Biggl[\dots,j(g_{i,j}\tau_i),\dfrac{j'(g_{i,j}\tau_i)z_{i,j}}{m_{g_{i,j}}(\tau_i)},p_{c_{i,j}}\left(j(g_{i,j}\tau_i),\chi^*(g_{i,j}\tau_i),\dfrac{j'(g_{i,j}\tau_i)z_{i,j}}{m_{g_{i,j}}(\tau_i)}\right),\dots,w_i,x_i,y_i,\dots\Biggr]\in V.\]

 If this holds for some choice of the possible data $z_{i,j}$, $c_{i,j}$, $w_i$, $x_i$, $z_i$, we simply say $B$ is adjacent to $V$, and write
 \[B\hookrightarrow V.\]
 Finally, if $G$ is a translate of a basic linear variety $B$ by $(\sigma_1,\dots,\sigma_d)$, we say $G$ is adjacent to $V$ if $B$ is adjacent to $V$ via any $z_{i,j}$, any $c_{i,j}$ and $(w_i,x_i,z_i)=J(\sigma_i)$.   In this case we again write
 \[G\hookrightarrow V.\]
\end{defn}
This rather intricate definition turns out to be crucial in carrying out a suitable variant of the usual o-minimal strategy used for Andr\'e-Oort problems.  This is the main new idea of the paper: we follow the typical Pila-Zannier strategy for diophantine problems of this type, but rather than counting those $\uh$-special varieties which are contained in $J^{-1}(V)$ directly, we instead count the $\uh$-special varieties which are adjacent to $V$.  The notion of adjacency is constructed so as to be definable, invariant under the action of $\slz$ (on the $g_{i,j}$) and invariant under Galois action (on $j$ and $\chi^*$); those three conditions are precisely what we need to make the strategy work.\\

The first step of the strategy and the primary goal of this this section is to prove an ``Ax-Lindemann-style'' result pertaining to the notion of adjacency.  The idea is that the only linear varieties (or more generally, real algebraic arcs) that can be adjacent to some variety $V$ should be accounted for by weakly $\uh$-special varieties.  Given our overall goal - to count those linear varieties which are adjacent to $V$ - this problem is directly analogous to the usual Ax-Lindemann theorems needed in the classical case.  Unsurprisingly, we will need to use some existing Ax-Lindemann results.

\begin{thrm}\label{thrm:AxLindwDerivs}
 Let $S\suq\uh^n$ be an arc of a real algebraic curve and let $G$ be the smallest weakly $\uh$-special variety containing $S$.  Suppose that $G$ is a GUT variety.  (Recall: this simply means that the matrices defining $G$ are upper triangular.)
 
 Let $V\suq\mathbb{C}^{3n+1}$ be an algebraic variety such that
 \[(\tau_1,J(\tau_1),\dots,J(\tau_n))\in V\]
 for all $(\tau_1,\dots,\tau_n)\in S$.  Then in fact this holds for all $(\tau_1,\dots,\tau_n)\in G$.
\end{thrm}

\begin{thrm}\label{thrm:NonholomorphicAxLind}
 Let $S\suq\uh^n$ be an arc of a real algebraic curve and let $G$ be the smallest weakly $\uh$-special variety containing $S$.
 Let $V\suq\mathbb{C}^{2n}$ be an algebraic variety such that
 \[(\dots,j(\tau_j),\chi^*(\tau_j),\dots)\in V\]
 for all $(\tau_1,\dots,\tau_n)\in S$.  Then in fact this holds for all $(\tau_1,\dots,\tau_n)\in G$.
\end{thrm}
Both of the above were proven in \cite{Spence2016}, though the majority of the work towards \ref{thrm:AxLindwDerivs} was done by Pila in \cite{Pila2013}.  Before we can apply these to prove our central Ax-Lindemann result, we will need some technical lemmas, the first of which is simply a strengthening of \ref{thrm:AxLindwDerivs}.

\begin{thrm}\label{thrm:AxLindAllFourHolomorphics}
 Let $S\suq\uh^n$ be an arc of a real algebraic curve and let $G$ be the smallest weakly $\uh$-special variety containing $S$.  Suppose that $G$ is a GUT variety.
 
 Let $V\suq\mathbb{C}^{4n+1}$ be an algebraic variety such that
 \[(\tau_1,j(\tau_1),j'(\tau_1),\chi(\tau_1),f(\tau_1),\dots,j(\tau_n),j'(\tau_n),\chi(\tau_n),f(\tau_n))\in V\]
 for all $(\tau_1,\dots,\tau_n)\in S$.  Then in fact this holds for all $(\tau_1,\dots,\tau_n)\in G$.
 \begin{proof}
  Let $F$ be a defining polynomial of $V$.  We will represent the various coordinates as follows.
  \begin{itemize}
   \item The $\tau_1$ coordinate will be represented by a variable $T$.
   \item The $j$-coordinates (ie. the 2nd, 6th, coordinates, etc.) will be represented by variables $J_1,\dots,J_n$.
   \item The $j'$-coordinates will be represented by variables $K_1,\dots,K_n$.
   \item The $\chi$-coordinates will be represented by variables $X_1,\dots,X_n$.
   \item The $f$-coordinates will be represented by variables $F_1,\dots,F_n$.
  \end{itemize}
  Since $j,j',\chi$ and $f$ are algebraically dependent, there is an irreducible polynomial $p$ with the property that
  \[p(j(\tau),j'(\tau),\chi(\tau),f(\tau))=0\]
  for all $\tau$.  
  Consider the variety $W\suq\mathbb{C}^{4n+1}$ defined by 
  \[p(J_{i},K_{i},X_{i},F_{i})=0\]
  for each $1\leq i\leq n$.  Clearly $\dim W=3n+1$.  
  
  If we further impose the condition
  \[F(X_1,\dots,X_{4n+1})=0,\]
  there are two possibilities.  Either the resulting variety $W_F$ still has dimension $3n+1$ or it has dimension $3n$.  
  
  If $\dim W_F=3n+1$, it is automatically the case that 
  \[F(\tau_1,\dots,j(\tau_k),j'(\tau_k),\chi(\tau_k),f(\tau_k),\dots)=0\]
  for all $(\tau_1,\dots,\tau_n)\in\uh^n$.
  On the other hand, if $\dim W_F=3n$, then $W_F$ amounts to the imposition of a relation between 3 of the 4 functions.  That is, there are distinct $A,B,C\in \{J,K,X,F\}$ and a polynomial $H$, in $3n+1$ variables, such that $W_F$ is defined by:
  \[(T,J_1,\dots F_{n})\in W\]
  and 
  \[H(T,A_1,B_1,C_1,\dots, A_n,B_n,C_n)=0.\]
  We then have, for the corresponding $f_A,f_B,f_C\in\{j,j',\chi,f\}$ that
  \[H(\tau_1,\dots,f_A(\tau_i),f_B(\tau_i),f_C(\tau_i),\dots)=0\]
  for all $(\tau_1,\dots,\tau_n)\in S$.  By Theorem \ref{thrm:AxLindwDerivs}, this must then hold for all $(\tau_1,\dots,\tau_n)\in G$.  In particular,
  \[F(\tau_1,\dots,j(\tau_i),j'(\tau_i),\chi(\tau_i),f(\tau_i),\dots)=0\]
  for all $(\tau_1,\dots,\tau_n)\in G$.
  
  This holds for each defining polynomial $F$ of $V$, so we're done. 
 \end{proof}
\end{thrm}

\begin{cor}\label{cor:ContinuationInAlgFunctions}
 Let $S\suq\uh^n$ be an arc of a real algebraic curve and let $\phi$ be an algebraic function in $4n+1$ variables.  Let $G$ be the smallest weakly $\uh$-special variety containing $S$, and suppose that $G$ is a GUT variety. 
 Writing 
 \[\tilde\pi(\tau_1,\dots,\tau_n)=(j(\tau_1),j'(\tau_1),\chi(\tau_1),f(\tau_1),\dots,j(\tau_n),j'(\tau_n),\chi(\tau_n),f(\tau_n)),\]
 we suppose that, on some branch of $\phi$,
 \[\phi(\tau_1,\tilde\pi(\tau))=0\]
 for all $\tau=(\tau_1,\dots,\tau_n)\in S$.  Then this holds for all $\tau\in G$, excluding perhaps some exceptional set corresponding to branch points of $\phi$.
 \begin{proof}
  There exists an irreducible polynomial $p$ such that
  \[p(\phi(\mathbf{X}),\mathbf{X})=0\]
  for all $\mathbf{X}$.  Then in particular we have
  \[p(0,\tau_1,\tilde\pi(\tau))=0\]
  for all $\tau=(\tau_1,\dots,\tau_n)\in S$.  By Theorem \ref{thrm:AxLindAllFourHolomorphics}, we have this relation for all $\tau\in G$.  
  
  We can pick a point $\mathbf{q}=(\tau_1,\dots,\tau_n)\in S$, a $G$-open neighbourhood $U$ of $\mathbf{q}$ and a $\mathbb{C}$-open neighbourhood $V$ of 0 with the following property.  Whenever $\tau\in U$, the only root of 
  \begin{equation}\label{eqn:pEquation}p(X,\tau_1,\tilde\pi(\tau))\end{equation}
  lying in $V$ is root 0 (which is a root by the earlier discussion).  Now, for all $\tau\in\uh^n$, we have
  \[p(\phi(\tau_1,\tilde\pi(\tau)),\tau_1,\tilde\pi(\tau))=0\]
  by definition of $p$.  In other words, $\phi(\tau_1,\tilde\pi(\tau))$ is a root of (\ref{eqn:pEquation}).  However, as $\tau\in U$ gets arbitrarily close to $\mathbf{q}$, the value of $\phi(\tau_1,\tilde\pi(\tau))$ gets arbitrarily close to 0.  Hence it eventually lies in $V$.  The only root of (\ref{eqn:pEquation}) within $V$ is 0, whence for some $G$-open neighbourhood of $\mathbf{q}$, we have
  \[\phi(\tau_1,\tilde\pi(\tau))=0.\]
  By analytic continuation, this holds for all $\tau\in G$, excluding some exceptional set corresponding to the branch points and branch cuts of $\phi$.  
 \end{proof}
\end{cor}

We conclude this section with one more technical lemma.  While it may appear entirely unmotivated, hopefully the need for such a lemma will become clear during the proof of Theorem \ref{propn:AxLindOverQ}.
\begin{lma}\label{lma:yIsAlgebraic}
 Let $S\suq\uh^n$ be an arc of a real algebraic curve and let $\phi$ be an algebraic function in $4n$ variables.  Suppose that
 \[\Imm\tau_1=\phi(\dots,j(\tau_k),j'(\tau_k),\chi(\tau_k),f(\tau_k),\dots)\]
 for all $(\tau_1,\dots,\tau_n)\in S$.  Let $G$ be the smallest weakly $\uh$-special variety containing $S$, and suppose that $G$ is a GUT variety.  Then $\Imm\tau_1$ is constant on $S$.
 \begin{proof}
  Write $y=\Imm\tau_1$ and suppose that $y$ is nonconstant.  Let us retain the abbreviation 
  \[\tilde\pi(\tau)=(j(\tau_1),j'(\tau_1),\chi(\tau_1),f(\tau_1),\dots,j(\tau_n),j'(\tau_n),\chi(\tau_n),f(\tau_n)).\]
  Then $S$ can be parametrised as
  \[S=\{(x_1(y)+iy,x_2(y)+iy_2(y),\dots,x_n(y)+iy_n(y)):y\in U\}\]
  for some interval $U\suq\mathbb{R}$ and algebraic functions $x_i$, $y_i$.  
  
  Now, take one of the polynomials $p(x_1,y_1,\dots,x_n,y_n)$ defining $S$.  We can write
  \begin{equation}\label{eqn:pT1}p\bigl[\tau_1-i\phi(\tilde\pi(\tau)),\phi(\tilde\pi(\tau)),x_2(\phi(\tilde\pi(\tau))),y_2(\phi(\tilde\pi(\tau))),\dots,x_n(\phi(\tilde\pi(\tau))),y_n(\phi(\tilde\pi(\tau)))\bigr]=0\end{equation}
  for all $\tau=(\tau_1,\dots,\tau_n)\in S$.  Otherwise put, we have an algebraic function $\psi$ such that
  \[\psi(\tau_1,\tilde\pi(\tau))=0\]
  for all $\tau\in S$.  By Corollary \ref{cor:ContinuationInAlgFunctions}, this holds for all $\tau\in G$, whence \ref{eqn:pT1} holds for all $\tau\in G$.\\
  
  By assumption, $G$ is a GUT variety.  Since $y$ is assumed to be nonconstant on $S$, the variable $\tau_1$ cannot be constant on $G$.  So up to permutation of coordinates $G$ looks like
  \[\{(\tau_1,g_2\tau_1,\dots,g_k\tau_1):\tau_1\in\uh\}\times H\]
  for some upper triangular matrices $g_i\in\gl$ and some GUT variety $H$.
  
  For any $\tau_1\in \uh$, $\tau'\in H$ and any $t\in\mathbb{Z}$, we then have
  \[\tau_t:=(\tau_1+t,g_2(\tau_1+t),\dots,g_k(\tau_1+t),\tau')\in G.\]
  Since the $g_i$ are upper triangular, we can find an integer $N$ with the following property.  For every $t\in\mathbb{Z}$, there exists $k\in\mathbb{Z}$ with
  \[g_i(\tau_1+tN)=k+g_i(\tau_1),\]
  for all $i$.  By the periodicity of $j$, $j'$, $\chi$ and $f$, it follows that
  \[\tilde\pi(\tau_{tN})=\tilde\pi(\tau_0)\]
  for all $t\in\mathbb{Z}$.  So for all $\tau=(\tau_1,\dots,\tau_n)\in G$ and all $t\in\mathbb{Z}$ we have
  \[p(\tau_1+tN-i\phi(\tilde\pi(\tau)),\phi(\tilde\pi(\tau)),x_2(\phi(\tilde\pi(\tau))),y_2(\phi(\tilde\pi(\tau))),\dots,x_n(\phi(\tilde\pi(\tau))),y_n(\phi(\tilde\pi(\tau))))=0.\]
  In particular, whenever $\tau=(x_1(y)+iy,x_2(y)+iy_2(y),\dots,x_n(y)+iy_n(y))\in S$, we have
  \[p(x_1(y)+tN,y,x_2(y),y_2(y),\dots,x_n(y),y_n(y))=0.\]
  This holds for \emph{every} polynomial $p$ defining $S$.  Since $S$ has only one real dimension, it must therefore be a horizontal line in the $\tau_1$ coordinate.  That is, $y$ is constant.  Contradiction!
 \end{proof}
\end{lma}

\subsection{Ax-Lindemann for Adjacency}
We can now prove our main Ax-Lindemann theorem.  The idea is to show that a real algebraic arc in $\uh^n$ which is `adjacent' to a variety $V$ (in a suitable sense) must be contained in a weakly $\uh$-special variety which is itself adjacent to $V$.
\begin{thrm}\label{propn:AxLindOverQ}
 Let $V\suq\mathbb{C}^{3n}$ be an algebraic variety.  Let $S$ be an arc of a real algebraic curve lying in $(\uh\times \slr)^n$.  Define 
 \[\widehat{S}=\{(g_1\tau_1,\dots,g_n\tau_n):(\tau_1,g_1,\dots,\tau_n,g_n)\in S\},\]
 and suppose that $\widehat{S}$ is positive-dimensional; that is, not all of the $g_j\tau_j$ are constant on $S$.
 
 Further suppose that, for some $c_i\in\mathbb{R}$, 
 \[\left[\dots,j(g_i\tau_i),\dfrac{j'(g_i\tau_i)}{m_{g_i}(\tau_i)},p_{c_i}\left(j(g_i\tau_i),\chi^*(g_i\tau_i),\dfrac{j'(g_i\tau_i)}{m_{g_i}(\tau_i)}\right),\dots\right]\in V\]
 for all $(\tau_1,g_1,\dots,\tau_n,g_n)\in S$.  
 
 Then there exists a weakly $\uh$-special variety $G$ with
 \[\widehat{S}\suq G\hookrightarrow V.\]

 \end{thrm}
 \textbf{Note.} The functions $j(g_j\tau_j)$, $\chi^*(g_j\tau_j)$ and $j'(g_j\tau_j)/m_{g_j}(\tau_j)$ are unaffected if we replace $g_j$ by $\gamma g_j$ for any $\gamma \in \slz$.  Hence we may assume that $G$, the smallest weakly $\uh$-special variety containing $\widehat{S}$, is a GUT variety.  This will be useful several times throughout.\\
 
  \textbf{Idea of Proof.} First, we attempt to parametrise the relevant algebraic arcs in terms of the imaginary part $y_j$ of one of the variables.  With suitable manipulations, we reach one of two outcomes: either a particular complex analytic relation involving just $j$, $\chi$, $f$ and $j'$ holds, or $y_j$ is equal to an algebraic function of $j$, $\chi$, $f$ and $j'$.  In the first case, the result comes fairly easily.  In the second case, we apply Lemma \ref{lma:yIsAlgebraic} to see that $y_j$ is constant; this situation can be dealt with easily.
 \begin{proof}[Proof of 2.8]
  Given $(\dots,\tau_j,g_j,\dots)\in S$, let us write 
  \[g_j\tau_j=\sigma_j=x_j+iy_j\] and 
  \[m_{g_j}(\tau_j)=\rho_j=u_j+iv_j.\]
  We wish to parametrise the real algebraic arc
  \[\tilde{S}=\{(\dots,g_j\tau_j,m_{g_j}(\tau_j),\dots):(\dots,\tau_j,g_j,\dots)\in S\}\suq (\uh\times\mathbb{C})^n\]
  in terms of one of the $y_j$.  Thus we first have to deal with the possibility that all of the $y_j$ are in fact constant on $\tilde S$.  
  
  In this situation, let us first assume that the $\rho_j$ are also constant.  Then we have
  \begin{equation}\label{eqn:ConstantCase}\left[\dots,j(\sigma_j),\dfrac{j'(\sigma_j)}{\rho_j},p_{c_j}\left(j(\sigma_j),\chi(\sigma_j)-\dfrac{3f(\sigma_j)}{\pi y_j},\dfrac{j'(\sigma_j)}{\rho_j}\right),\dots\right]\in V\end{equation}
  for some constants $y_j$, $\rho_j$ and all $\sigma=(\sigma_1,\dots,\sigma_n)\in\widehat{S}$.  Recall that $G$, the weakly $\uh$-special closure of $\widehat{S}$, can be assumed to be a GUT variety.  So we can apply Theorem \ref{thrm:AxLindAllFourHolomorphics} to see that (\ref{eqn:ConstantCase}) holds for all $\sigma\in G$.  
  
  Taking Zariski closures (over $\mathbb{C}$) in (\ref{eqn:ConstantCase}), we get
  \[\left[\dots,j(\sigma_j),w_j,p_{c_j}\left(j(\sigma_j),\chi(\sigma_j)-\dfrac{3f(\sigma_j)}{\pi y_j},w_j\right),\dots\right]\in V\]
  for all $\mathbf{q}\in G$ and $(w_1,\dots,w_n)\in V_{j(\mathbf{q})}$.  Here $V_{j(\mathbf{q})}$ is a variety depending only on the $j(\sigma_j)$, which contains the point 
  \[\left(\dfrac{j'(\sigma_1)}{\rho_1},\dots,\dfrac{j'(\sigma_n)}{\rho_n}\right).\]  
  As in \cite[Lemma 4.10]{Spence2016}, it is easy to find a sequence of matrices $\gamma_t\in\slz$, $t\in\mathbb{N}$, such that 
  \[(\chi(\gamma_t\sigma_1),\dots,\chi(\gamma_t\sigma_n))\to (\chi(\sigma_1),\dots,\chi(\sigma_n))\]
  and
  \[(f(\gamma_t\sigma_1),\dots,f(\gamma_t\sigma_n))\to 0\]
  as $t\to\infty$.  So by continuity (and the invariance of $j$), we get
  \[\left[\dots,j(\sigma_j),w_j,p_{c_j}\left(j(\sigma_j),\chi(\sigma_j),w_j\right),\dots\right]\in V\]
  for all $\mathbf{q}\in G$ and $\mathbf{w}\in W_{j(\mathbf{q})}$.  By an isomorphism theorem from \cite{Spence2016}, we get
  \[\left[\dots,j(\sigma_j),w_j,p_{c_j}\left(j(\sigma_j),\chi^*(\sigma_j),w_j\right),\dots\right]\in V\]
  and hence
  \[\left[\dots,j(\sigma_j),\dfrac{j'(\sigma_j)}{\rho_j},p_{c_j}\left(j(\sigma_j),\chi^*(\sigma_j),\dfrac{j'(\sigma_j)}{\rho_j}\right),\dots\right]\in V,\]
  for all $\sigma\in G$.  This says precisely that $G\hookrightarrow V$.

  Next we deal with the situation where one of the $y_j$ is nonconstant on $\tilde S$.  Without loss of generality, suppose it is $y_1$ and write $y=y_1$.  We can parametrise
  \[\tilde{S}=\{(x_1(y)+iy,u_1(y)+iv_1(y),\dots,x_n(y)+iy_n(y),u_n(y)+iv_n(y)):y\in I\},\]
  for some interval $I\suq\mathbb{R}$ and algebraic functions $x_i,y_i,u_i,v_i$.  Letting $F$ be a defining polynomial of $V$, we have
  \[F\left[\dots,j(\sigma_j),\dfrac{j'(\sigma_j)}{u_j(y)+iv_j(y)},p_{c_j}\left(j(\sigma_j),\chi(\sigma_j)-\dfrac{3f(\sigma_j)}{\pi y_j(y)},\dfrac{j'(\sigma_j)}{u_j(y)+iv_j(y)}\right),\dots\right]=0,\]
  where we set $y_1(y)=y$.  We can rewrite this; there is an algebraic function $s$ such that the above holds if and only if
  \[s(y,\tilde\pi(\sigma))=0\]
  for all $\sigma=(\sigma_1,\dots,\sigma_n)\in\widehat{S}$.  Here we are writing $y=\Imm\sigma_1$ and, as before:
  \[\tilde\pi(\sigma_1,\dots,\sigma_n)=(j(\sigma_1),j'(\sigma_1),\chi(\sigma_1),f(\sigma_1),\dots,j(\sigma_n),j'(\sigma_n),\chi(\sigma_n),f(\sigma_n)).\]
  
  Since $s$ is an algebraic function, there is a nontrivial irreducible polynomial $p_s$ such that 
  \[p_s(s(\mathbf{X}),\mathbf{X})=0\]
  for all $\mathbf{X}$.  In particular,
  \[p_s(0,y,\tilde\pi(\sigma))=0\]
  for all $\sigma\in\widehat{S}$ and $y=\Imm\sigma_1$.  Since $p_s$ is irreducible and nontrivial, we get a nontrivial $q_s(\mathbf{X})=p_s(0,\mathbf{X})$.  (It is clear that $p_s(t,\mathbf{X})\ne t$.)
  
  Now we apply the following iterative procedure to $q_s$.
  \begin{enumerate}
   \item Inspect separately each coefficient $r_k$ of $T^k$ in $q(T,\dots)$.  If 
    \[r_k(\tilde\pi(\sigma))=0\]
    for all $\sigma\in\widehat{S}$ and all $k$, then terminate.  Otherwise, let $q'$ be the polynomial produced by removing from $q$ all coefficients $r_k$ which have the above property.
   \item If $q'$ is irreducible, terminate.  Otherwise, there is a factor $q''$ of $q'$ with
    \[q''(y,\tilde\pi(\sigma))=0\]
    for all $\sigma\in\widehat{S}$ and $y=\Imm\sigma_1$.  
   \item We have a polynomial $q''$, which retains the property that
   \[q''(y,\tilde\pi(\sigma))=0\]
    for all $\sigma\in\widehat{S}$ and $y=\Imm\sigma_1$.  Repeat from step 1, with $q''$ instead of $q$.   
  \end{enumerate}
  This must eventually terminate, since step 2 will always reduce the degree of the polynomial in question.  So we have two possibilities.\\
  
  If we terminated at step 1, then working backwards we see that every coefficient $r_k$ of $T^k$ in $q_s$ has the property that
  \[r_k(\tilde\pi(\sigma))=0\]
  for all $\sigma\in\widehat{S}$.  Using the fact that $G$ is a GUT variety, we can apply Theorem \ref{thrm:AxLindAllFourHolomorphics} to see that this holds for all $\sigma\in G$.  In particular, 
  \[q_s(y,\tilde\pi(\sigma))=0\]
  for all $y\in\mathbb{C}$ and all $\sigma\in G$.  Otherwise put, 0 is a root of
  \begin{equation}\label{eqn:pSEquation}p_s(X,y,\tilde\pi(\sigma))\end{equation}
  for all $y\in\mathbb{C}$ and all $\sigma\in G$.
  
  Now we proceed much as we did in the proof of Corollary \ref{cor:ContinuationInAlgFunctions}.  Choose a point 
  \[\mathbf{a}=(a_1,\dots,a_n)\in\widehat{S},\]
  a $G$-open neighbourhood $U$ of $\mathbf{a}$ and a $\mathbb{C}$-open neighbourhood $W$ of 0 with the following property.  Whenever $\sigma=(\sigma_1,\dots,\sigma_n)\in U$, the only root of (\ref{eqn:pSEquation}) lying in $W$ is 0 itself.
  
  Now recall that 
  \[s(y,\tilde\pi(\sigma))\] 
  is a root of (\ref{eqn:pSEquation}) for all $y$ and all $\sigma$.  This is just the definition of $p_s$.  Since $s$ vanishes at $\mathbf{a}\in\widehat{S}$, there is a $G$-open neighbourhood 
  \[\mathbf{a}\in U'\suq U\]
  such that
  \[s(\Imm a_1,\tilde\pi(\sigma))\in W\]
  for all $\sigma\in U'$.  But the only root of (\ref{eqn:pSEquation}) lying in $W$ is the root 0.  So it must be the case that
  \[s(\Imm a_1,\tilde\pi(\sigma))=0\]
  for all $\sigma\in U'$ and hence for all $\sigma\in G$.  
  
  Recalling the definition of $S$, we see that
  \[F\left[\dots,j(\sigma_j),\dfrac{j'(\sigma_j)}{u_j(y)+iv_j(y)},p_{c_j}\left(j(\sigma_j),\chi(\sigma_j)-\dfrac{3f(\sigma_j)}{\pi y_j(y)},\dfrac{j'(\sigma_j)}{u_j(y)+iv_j(y)}\right),\dots\right]=0,\]
  for all $\sigma=(\sigma_1,\dots,\sigma_n)\in G$ and for \emph{constants} $y_j=y_j(\Imm a_1)$ and $\rho_j=u_j(\Imm a_1)+iv_j(\Imm a_1)$.  
  
  If we repeat this whole procedure for each defining polynomial of $V$, we get
  \[\left[\dots,j(\sigma_j),\dfrac{j'(\sigma_j)}{\rho_j},p_{c_j}\left(j(\sigma_j),\chi(\sigma_j)-\dfrac{3f(\sigma_j)}{\pi y_j},\dfrac{j'(\sigma_j)}{\rho_j}\right),\dots\right]\in V\]
  for all $\sigma\in G$ and \emph{constants} $y_j=y_j(\Imm a_1)$ and $\rho_j=u_j(\Imm a_1)+iv_j(\Imm a_1)$.  Now we are in exactly the same position as we were for (\ref{eqn:ConstantCase}), so we conclude as we did earlier.\\
  
  If we terminated at step 2, then we have an irreducible polynomial $q$ such that
  \[q(y,\tilde\pi(\sigma))=0\]
  for all $\sigma=(\sigma_1,\dots,\sigma_n)\in \widehat{S}$.  Moreover, for every $k$, the coefficient $r_k$ of $T^k$ in $q(T,\dots)$ has the property that
  \[r_k(\tilde\pi(\sigma))\]
  does \emph{not} vanish identically for $\sigma\in\widehat{S}$.  Thus we can extract an algebraic function $\phi$ such that
  \[y=\phi(\tilde\pi(\sigma))\]
  for all $\sigma\in \widehat{S}$ and $y=\Imm\sigma_1$.  Now we may apply Lemma \ref{lma:yIsAlgebraic} (once again using the fact that the weakly $\uh$-special closure of $\widehat{S}$ is a GUT variety) and see that $y$ is constant on $\widehat{S}$, a contradiction.\\
  
  Finally we deal with the case where the $y_j$ are all constant on $\tilde S$, but perhaps the $\rho_j$ vary.  Say $y_j=a_j$.  Note, by hypothesis, that at least one of the $\sigma_j$ is nonconstant on $\tilde S$.  Without loss of generality, let us say it is $\sigma_1$.
  
  We have the relation
  \begin{equation}\label{eqn:constantYs}
   \left[\dots,j(\sigma_j),\dfrac{j'(\sigma_j)}{\rho_j},p_{c_j}\left(j(\sigma_j),\chi(\sigma_j)-\dfrac{3f(\sigma_j)}{\pi a_j},\dfrac{j'(\sigma_j)}{\rho_j}\right),\dots\right]\in V,   
  \end{equation}
  holding for all $(\dots,\sigma_j,\rho_j,\dots)\in \tilde S$.  This is a complex analytic relation.  
  
  Consider the intersection of an irreducible algebraic variety $W\suq\mathbb{C}^{2n}$ with $(\uh\times\mathbb{C})^n$.  A connected component of this intersection is called a complex algebraic  component.  Let $A$ be the smallest complex algebraic component containing $\tilde S$.  Then by analytic continuation, (\ref{eqn:constantYs}) holds for all $(\dots,\sigma_j,\rho_j,\dots)\in A$.
  
  Since the weakly $\uh$-special closure of $\widehat{S}$ is a GUT variety $G$, the projection of $A$ onto the $\sigma_j$ coordinates also has $G$ as its weakly $\uh$-special closure.  Hence we can find a real algebraic arc
  \[T\suq (\uh\times\mathbb{C})^n\]
  with the following properties:
  \begin{itemize}
   \item On $T$, the imaginary part of $\sigma_1$ is nonconstant.
   \item Whenever $(\dots,\sigma_j,r_j,\dots)\in T$, we have
   \[
    \left[\dots,j(\sigma_j),\dfrac{j'(\sigma_j)}{\rho_j},p_{c_j}\left(j(\sigma_j),\chi(\sigma_j)-\dfrac{3f(\sigma_j)}{\pi a_j},\dfrac{j'(\sigma_j)}{\rho_j}\right),\dots\right]\in V.   
   \]
   \item The projection of $T$ onto the $\sigma_j$ coordinates has $G$ as its weakly $\uh$-special closure.
  \end{itemize}
  We can then parametrise $T$ in terms of $y=\Imm \sigma_1$.  Exactly as before, we then rewrite the polynomial relation as an algebraic function $s$, yielding
  \[s(y,\tilde\pi(\sigma_1,\dots,\sigma_n))=0\]
  whenever $(\dots,\sigma_j,r_j,\dots)\in T$ and $y=\Imm \sigma_1$.  By the same analysis as earlier, we end up with a constant $a$ such that
  \[s(a,\tilde\pi(\sigma_1,\dots,\sigma_n))=0\]
  for all $(\sigma_1,\dots,\sigma_n)\in G$.  This in turn yields
  \[\left[\dots,j(\sigma_j),\dfrac{j'(\sigma_j)}{\rho_j(a)},p_{c_j}\left(j(\sigma_j),\chi(\sigma_j)-\dfrac{3f(\sigma_j)}{\pi a_j},\dfrac{j'(\sigma_j)}{\rho_j(a)}\right),\dots\right]\in V\]
  for all $(\sigma_1,\dots,\sigma_n)\in G$.  Since $\rho_j(a)$ and $a_j$ are constants, we can then conclude exactly as we did for (\ref{eqn:ConstantCase}).
 \end{proof}

The above theorem will be used in our point-counting arguments in the next section, in order to take a real algebraic arc and produce from it a weakly $\uh$-special variety with certain adjacency properties.  We will also need the following corollary, which we will eventually use to ensure that there are only finitely many basic linear varieties adjacent to our fixed variety $V$.

\begin{cor}\label{cor:BasicsAreSpecial}
 Let $V\suq\mathbb{C}^{3k}$ be a variety and let $B$ be a basic linear variety adjacent to $V$.  Suppose $B$ is maximal with this property.  Then $B$ is $\uh$-special, ie. all of the $g_{i,j}$ defining $B$ lie in $\gl$.
 \begin{proof}
  Immediate from \ref{propn:AxLindOverQ}.
 \end{proof}

\end{cor}

\section{Point-Counting}
In this section, we will discuss the necessary o-minimality and point-counting considerations.  For familiar readers, the results here should fit well with expectations, though they are necessarily rather less neat than their equivalents in more classical settings.  We are assuming some basic familiarity with o-minimality and the Pila-Wilkie theorems; see \cite{Pila2006}, \cite{Dries1998} and \cite{Dries1994}.

The first crucial fact we need is that all of the relevant functions are definable in an appropriate sense.  Namely, the restrictions of $j$, $j'$, $j''$, $\chi^*$, $\chi$ and $f$ to any standard fundamental domain for the action of $\slz$ on $\uh$ are definable in (the o-minimal structure) $\Raxp$.  We will use the ``most standard'' fundamental domain
\[\mathbb{F}=\{\tau\in\uh:-1/2\leq\operatorname{Re}\tau\leq 1/2, |\tau|\geq 1\}.\]
The definability can be seen in various ways; either as a consequence of the theory of elliptic curves and a result of Peterzil and Starchenko on definability of the Weierstrass $\wp$-function \cite{Peterzil2004}, or via $q$-expansions.   Given this fact, the idea is a fairly standard one:

First assume that a variety contains an infinite set of special points.  Then the preimage of that variety (under $J$ for instance), intersected with $\mathbb{F}^n$, is a definable set in $\Raxp$.  By taking Galois conjugates, one can force this definable set to contain many quadratic points of bounded height, in the sense of the Pila-Wilkie theorem.  So it will contain a real algebraic arc, allowing us to apply the Ax-Lindemann results of the previous section.

The missing ingredient so far is the Galois aspect.  Fortunately, much of this is already done for us; the points about which we need Galois information are just the $j$-special and $\chi^*$-special points.  For control over these, we have the following.

\begin{propn}\label{propn:GaloisOrbits}
 Let $\tau\in\uh$ be a quadratic point and consider the algebraic numbers $j(\tau)$ and $\chi^*(\tau)$.  Let $\sigma$ be a Galois conjugation acting on $\mathbb{Q}(j(\tau))\supseteq\mathbb{Q}(\chi^*(\tau))$.  Let $\tau'$ be a quadratic point such that $j(\tau')=\sigma(j(\tau))$.  Then $\chi^*(\tau')=\sigma(\chi^*(\tau))$.
\end{propn}
For a proof, see \cite[Proposition 5.2]{Spence2016}.  This tells us, essentially, that to keep track of the Galois conjugates of $\chi^*(\tau)$, we need only keep track of the Galois conjugates of $j(\tau)$.  We already have sufficiently good control of the Galois conjugates of $j(\tau)$; it is a consequence of the Siegel bound for class numbers of quadratic fields \cite{Siegel1935} that
\begin{equation}\label{eqn:SiegelBound}[\mathbb{Q}(j(\tau)):\mathbb{Q}]\gg D^{\frac{1}{4}},\end{equation}
where $D$ is the discriminant of $\tau$.  (In fact we can do much better, but this is sufficient for our purposes.)  Hence in particular, there will be $\gg D^{\frac{1}{4}}$ Galois conjugates of a point $(j(\tau),\chi^*(\tau))$, over any fixed number field.  This fact is central to our main ``point-counting theorem''.  In this theorem we use the assumption of $j'$-genericity for the first, and only, time in the paper.

\begin{thrm}\label{thrm:MainPointCountingThrm}
 Let $V\suq\mathbb{C}^{3n}$ be an algebraic variety defined over $\overline{\mathbb{Q}}$.  Then there is a number $D=D(V)$ with the following property.  Let $\tau\in J^{-1}(V)$ be a $j'$-generic quadratic point with discriminant greater than $D$, and suppose none of the coordinates of $\tau$ lies in $\slz\cdot\{i,\rho\}$.  Then there is an $\uh$-special variety $G$ with
 \[\tau\in G \hookrightarrow V.\]
 \begin{proof}
  Let $K$ be a number field containing a field of definition for $V$.  
  
  Suppose we have a partition of $\{1,\dots, n\}$,
  \[S_1\cup\dots\cup S_k,\]
  with each $S_i\ne\emptyset$.  For each $i$, let $s_i=\min S_i$ and $r_i=\#S_i-1$.  Given 
  \[\sigma=(\sigma_1,\dots,\sigma_k)\in\uh^k\]
  and 
  \[g=(g_{1,1},\dots,g_{1,r_1},\dots,g_{k,1},\dots,g_{k,r_k})\in\slr^{n-k},\]
  define the following set:
  \begin{multline*}
  Z_{\sigma,g}=\Bigg\{(\tau,h)\in\uh^k\times\glr^{n-k}: \det h_{i,j}=\det g_{i,j},\\
   \bigg[\dots, j(\tau_i),j'(\tau_i),p_{\Imm\sigma_i}(j(\tau_i),\chi^*(\tau_i),j'(\tau_i)),\dots\\
   \dots, j(h_{i,j}\tau_i),j'(h_{i,j}\tau_i)\dfrac{m_{g_{i,j}}(\sigma_i)}{m_{h_{i,j}}(\tau_i)},p_{\Imm g_{i,j}\sigma_i}\left(j(h_{i,j}\tau_i),\chi^*(h_{i,j}\tau_i),j'(h_{i,j}\tau_i)\dfrac{m_{g_{i,j}}(\sigma_i)}{m_{h_{i,j}}(\tau_i)}\right),\dots\bigg]\in V\Bigg\}   
  \end{multline*}
  Consider this as a family of sets, fibred over $\uh^k\times\slr^{n-k}$.  There is one such family for each of the finitely many partitions of $\{1,\dots, n\}$, and we consider them all together.  They are certainly \emph{not} definable families.  However, for a given partition, the family
  \[\mathcal{Z}_{\sigma,g}=\{(\tau,h)\in Z_{\sigma,g}:\tau_i, h_{i,j}\tau_i\in \mathbb{F}, i\leq k, j\leq r_k\}\]
  is definable in $\Raxp$.\\
  
  Now let us consider a $j'$-generic special point $\tau\in J^{-1}(V)$, of large discriminant $D(\tau)$.  Up to permutation of coordinates, $\tau$ looks like
  \[\tau=(\sigma_1,g_{1,1}\sigma_1,\dots,g_{1,r_1}\sigma_1,\dots,\sigma_k,g_{k,1}\sigma_k,\dots,g_{k,r_k}\sigma_k),\]
  with the $\sigma_i$ lying in distinct $\gl$-orbits, and $g_{i,j}\in\gl$.  So $\tau$ corresponds to a partition of $\{1,\dots, n\}$ in the obvious way.  Writing $g_{i,j}$ as primitive integer matrices, let $N_{i,j}=\det g_{i,j}$.  Recall that the $j'$-genericity of $\tau$ means that
  \[j'(\sigma_1),\dots,j'(\sigma_k)\]
  are algebraically independent over $\overline{\mathbb{Q}}$.  So we see that the $\overline{\mathbb{Q}}$-Zariski closure $\JClose{\tau}$ of $J(\tau)$ is the set of points of the form
  \begin{multline*}\Biggl[\dots,j(\sigma_i),w_i,p_{\Imm\sigma_i}(j(\sigma_i),\chi^*(\sigma_i),w_i),\dots,j(g_{i,j}\sigma_i),-w_i\lambda_{N_{i,j}}\bigl(j(\sigma_i),j(g_{i,j}\sigma_i)\bigr)m_{g_{i,j}}(\sigma_i),\\p_{\Imm g_{i,j}\sigma_i}\biggl(j(g_{i,j}\sigma_i),\chi^*(g_{i,j}\sigma_i),-w_i\lambda_{N_{i,j}}\bigl(j(\sigma_i),j(g_{i,j}\sigma_i)\bigr)m_{g_{i,j}}(\sigma_i)\biggr),\dots\Biggr],\end{multline*}
  for some $w_1,\dots,w_k\in\mathbb{C}$.
  
  We will show that the existence of this $\tau$ implies that $\mathcal{Z}_{\sigma,g}$ contains $\gg D(\tau)^{\frac{1}{4}}$ quadratic points of bounded height.  As is typical in the Pila-Zannier strategy, these new points will arise from Galois conjugates of $j(\tau)$.  To begin, we need to define a variety which keeps track of $\tau$ and its Galois conjugates.  Said variety will be a subvariety of $\mathbb{C}^{2n}$; we will write a general element of $\mathbb{C}^{2n}$ as 
  \[(\dots,X_i,Y_i,\dots,X_{i,j},Y_{i,j},\dots),\]
  with $i\leq k$ and $j\leq r_i$, matching the structure of the underlying partition of $\{1,\dots,n\}$.
  
  Let
  \begin{multline*}V_{\sigma,g}=\Biggl\{(\mathbf{X},\mathbf{Y})\in\mathbb{C}^{2n}:\forall w_1,\dots,w_k\in\mathbb{C},\\ \Bigl[\dots,X_i,w_i,p_{\Imm\sigma_i}(X_i,Y_i,w_i),\dots,X_{i,j},-w_i\lambda_{N_{i,j}}(X_i,X_{i,j})m_{g_{i,j}}(\sigma_i),\\p_{\Imm g_{i,j}\sigma_i}\bigl(X_{i,j},Y_{i,j},-w_i\lambda_{N_{i,j}}(X_i,X_{i,j})m_{g_{i,j}}(\sigma_i)\bigr),\dots\Bigr]\in V\Biggr\}.\end{multline*}
  Then $V_{\sigma,g}$ is a subvariety of $\mathbb{C}^{2n}$, defined over $K(\sigma,\Imm\sigma)$.  This definition is set up to mirror the shape of $\JClose{\tau}$.  Thus, since $\JClose{\tau}\suq V$, we see that $V_{\sigma,g}$ must contain the point
  \[(j,\chi^*)(\tau).\]
  Hence $V_{\sigma,g}$ also contains every Galois conjugate (over $K(\sigma,\Imm\sigma)$) of $(j,\chi^*)(\tau)$.  By Proposition \ref{propn:GaloisOrbits}, such a Galois conjugate must take the form $(j,\chi^*)(\tau')$, for some quadratc $\tau'$ with $D(\tau')=D(\tau)$.  Moreover, by the existence of the modular polynomial $\Phi_N$, $\tau'$ must have the same $\gl$ structure as $\tau$.  That is:
  \[\tau'=(\sigma_1',g_{1,1}'\sigma_1',\dots,g_{1,r_1}'\sigma_1',\dots,\sigma_k',g_{k,1}'\sigma_k',\dots,g_{k,r_k}'\sigma_k'),\]
  where the $\sigma_i'$ are quadratic points and $g_{i,j}'= g_{i,j}\gamma_{i,j}$, for some $\gamma_{i,j}\in\slz$.  Further, by the modularity of $j$ and $\chi^*$, we can ensure that $\sigma_i'\in\mathbb{F}$.
  
  
  For each $\tau'$ arising this way, let us take $w_i=j'(\sigma_i')$ in the definition of $V_{\sigma,g}$.  Noting that
  \[-j'(\sigma_i')\lambda_{N_{i,j}}(j(\sigma_i'),j(g_{i,j}'\tau))m_{g_{i,j}}(\sigma_i)=j'(g_{i,j}'\sigma_i')\dfrac{m_{g_{i,j}}(\sigma_i)}{m_{g_{i,j}'}(\sigma_i')},\]
  we see that $\left(\sigma',g'\right)\in Z_{\sigma,g}$.  Further, there is $\gamma_{i,j}'\in\slz$ such that $\gamma_{i,j}'g_{i,j}'\sigma_i'\in\mathbb{F}$.  This yields $(\sigma',\gamma'g')\in\mathcal{Z}_{\sigma,g}$.
    
  By (\ref{eqn:SiegelBound}), there are $\gg D(\tau)^{1/4}$ Galois conjugates of $(j,\chi^*)(\tau)$ over $\mathbb{Q}$.  Since $K(\sigma,\Imm\sigma)$ is an extension of $K$ of degree at most $4n$, we have
  \[[\mathbb{Q}(j(\tau)):K(\sigma,\Imm\sigma)]=[\mathbb{Q}(j(\tau)):\mathbb{Q}]/c,\]
  where $c$ is an absolute constant.  Hence there are $\gg D(\tau)^{\frac{1}{4}}$ points $(\sigma',\gamma'g')$ lying in $\mathcal{Z}_{\sigma,g}$.
  
  Moreover, it is a consequence of Proposition 5.2 in \cite{Pila2011} that the corresponding $\gamma_{i,j},\gamma_{i,j}'$ can be chosen to have height polynomial in $D$, whence $(\sigma',\gamma' g')$ has height polynomial in $D$.  So the existence of $\tau\in J^{-1}(V)$, with discriminant $D(\tau)$, ensures that $\mathcal{Z}_{\sigma,g}$ contains $\gg D(\tau)^{1/4}$ points of bounded height (and degree at most 2).
  
  At this point, we can apply the uniform Pila-Wilkie Theorem.  Playing the upper bound from uniform Pila-Wilkie against the lower bound found above, we find a number $D$ such that whenever $\tau\in J^{-1}(V)$ has discriminant greater than $D$, the corresponding $\mathcal{Z}_{\sigma,g}$ contains an arc $T$ of a real algebraic curve.  Further, we can ensure that $T$ contains the $(\sigma',\gamma'g')$ corresponding to one of the $\tau'$ arising from the Galois conjugates of $(j,\chi^*)(\tau)$.
  
  We would like to apply Theorem \ref{propn:AxLindOverQ} to some algebraic arc constructed from $T$.  Indeed, let 
  \[S=\left\{\left(\dots,\tau_i,\begin{pmatrix}1&0\\0&1\end{pmatrix},\dots,\tau_i,\overline{h_{i,j}},\dots\right):(\tau,h)\in T\right\},\]
  where $\overline{h}$ is the element of $\slr$ corresponding to the image of $h$ as an element of $\operatorname{PGL}_2(\mathbb{R})$.  Also let
  \[\widehat{S}=\{(\dots,\tau_i,\dots,h_{i,j}\tau_i,\dots):(\tau,h)\in T\}.\]
  Before we can apply Theorem \ref{propn:AxLindOverQ}, it only remains to check that $\widehat{S}$ is indeed an arc, rather than just point.  This is easy to see; if $\tau_i$ and $h_{i,j}\tau_i$ are all constant on $\widehat{S}$, then $h$ must be constant, up to determinant, on $T$.  Since the determinant of $h_{i,j}$ is fixed in the definition of $Z_{\tau_0,g_0}$, it follows that $\tau_i$ and $h_{i,j}$ are both constant on $T$, whence $T$ itself is just a point, which is a contradiction.
  
  So \ref{propn:AxLindOverQ} yields an $\uh$-special set $H$ with
  \[\widehat{S}\suq G \hookrightarrow V.\]
  Since $(\sigma',\gamma'g')\in T$, we have $\gamma\tau'\in\widehat{S}$ for some $\gamma\in\slz^n$.  Hence, some $\slz$-translate $H'$ of $H$ contains $\tau'$, and remains adjacent to $V$.  Suppose that
  \[H'=B\times \{\tau_{k+1}',\dots,\tau_n'\}\]
  for some basic $\uh$-special variety $B$.  Since $(j,\chi^*)(\tau')$ was a Galois conjugate of $(j,\chi^*)(\tau)$, we can now apply the inverse Galois conjugation to see that
  \[B\times \{\tau_{k+1},\dots,\tau_n\}\hookrightarrow V.\]
  For suitable $\gamma$, the $\uh$-special variety
  \[G=\gamma B\times \{\tau_{k+1},\dots,\tau_n\}\]
  will contain $\tau$ and is still adjacent to $V$.  
 \end{proof}

\end{thrm}

The above result is one of two crucial pieces of ``counting'' we need in order to prove \ref{thrm:PartialAOwDerivs}.  The other half is the following proposition.  The idea is that \ref{thrm:MainPointCountingThrm} will be used to count the number of isolated points that can arise in $J^{-1}(V)$ (the``zero-dimensional pieces''), and this next proposition will count the ``positive-dimensional pieces;'' namely the basic $\uh$-special varieties.  Together, \ref{thrm:MainPointCountingThrm} and \ref{propn:BasicSpecialCount} will act as the engine driving the inductive argument at the heart of the proof of \ref{thrm:PartialAOwDerivs}.
\begin{propn}\label{propn:BasicSpecialCount}
 Let $V\suq\mathbb{C}^{3n}$ be a variety.  Consider the definable subset consisting of those \emph{proper} basic linear varieties $B\suq \uh^k$ (for any $k$) such that:
 \begin{itemize}
  \item $B$ meets $\mathbb{F}^k$ in its full dimension.
  \item $B$ is adjacent to $V$.
  \item $B$ is maximal with the above properties.
 \end{itemize}
 Then there are only finitely many such $B$, and each is $\uh$-special.
 \begin{proof}
  By Corollary \ref{cor:BasicsAreSpecial}, the only basic linear $B$ which are maximally adjacent to $W$ are necessarily $\uh$-special.  For such $B$, all the defining $g_{i,j}$ are in $\gl$, so the collection of such $B$ is parametrised by a countable set.  
  The conditions specified are definable, (compare with, for instance, \cite[Proposition 10.2]{Pila2011}) so we have a countable definable set, which is therefore finite.
 \end{proof}

\end{propn}

\section{Bringing the Proof Together}
As we bring everything together to prove our central theorem, let us recall the statement.

\begin{thrmREP}
 Let $V\suq\mathbb{C}^{3n}$ be a proper algebraic variety defined over $\overline{\mathbb{Q}}$.  There exists an $\slz$-finite collection $\sigma(V)$, consisting of \emph{proper} $\uh$-special varieties of $\uh^n$, with the following property.  Every $j'$-generic $\uh$-special point in $J^{-1}(V)$ is contained in some $G\in\sigma(V)$.
 \begin{proof}
  First, let us note: we can safely ignore any $\tau\in J^{-1}(V)$ which have a coordinate lying in $\slz\cdot\{i,\rho\}$, since these all clearly lie in $\slz$-finitely many proper $\uh$-special subvarieties.\\
  
  We work by induction on $n$.  For $n=1$, we argue as follows.  
  
  Suppose that $V\suq \mathbb{C}^3$ is an algebraic variety defined over $\overline{\mathbb{Q}}$ and let $D=D(V)$ be the number given to us by \ref{thrm:MainPointCountingThrm}.  If $J^{-1}(V)$ contains $\slz$-infinitely many quadratic points, then in particular it contains one of discriminant greater than $D$.  Theorem \ref{thrm:MainPointCountingThrm} then tells us that $\uh$ is adjacent to $V$.  This implies that
  \[\forall\tau\in\uh, (j(\tau),j'(\tau)z,p_c(j(\tau),\chi^*(\tau),j'(\tau)z))\in V\]
  for some $c>0$ and $z\in\mathbb{C}$.  Taking Zariski closures (over $\mathbb{C}$) this says that
  \[\forall w\in\mathbb{C},\tau\in\uh, (j(\tau),w,p_c(j(\tau),\chi^*(\tau),w))\in V.\]
  In particular, for any $\tau$ with $\Imm\tau=c$, we have
  \[(j(\tau),j'(\tau),p_c(j(\tau),\chi^*(\tau),j'(\tau)))\in V.\]
  For such $\tau$, $p_c(j(\tau),\chi^*(\tau),j'(\tau))=j''(\tau)$ by definition.  Hence
  \[J(\tau)\in V\]
  for all $\tau$ with $\Imm\tau=c$.  By analytic continuation, this says that $J(\uh)\suq V$, whence $V=\mathbb{C}^3$.\\
  
  So by induction we may assume the result holds for all $V\suq\mathbb{C}^{3k}$, $k<n$.

The first stage is to construct a variety $V^*$ which is designed to account for all possible positive-dimensional special subvarieties of $V$.  

Let $\mathcal{G}$ be the finite collection of proper basic $\uh$-special subvarieties (of some $\uh^k$) afforded by applying Proposition \ref{propn:BasicSpecialCount} to $V$.   Then let
\[\mathcal{G}_1^*=\Bigl\{\omega\bigl(\gamma\cdot(B\times\uh^{n-k})\bigr):B\in \mathcal{G}, B\suq\uh^k, \gamma\in\slz^n,\omega\text{ a permutation of the coordinates}\Bigr\}.\]
Since $\mathcal{G}$ was finite, $\mathcal{G}_1^*$ is $\slz$-finite.

Next, consider the variety $V_k\suq\mathbb{C}^{3(n-k)}$, defined over $\overline{\mathbb{Q}}$ by 
\begin{multline*}V_k=\{\mathbf{X}\in\mathbb{C}^{3(n-k)}:\text{The translate of }\mathbb{C}^{3k}\text{ by }\mathbf{X}\\\text{ (for some choice of ordering of coordinates) is contained in }V\}.\end{multline*}
Clearly $V_k$ is a proper subvariety of $\mathbb{C}^{3(n-k)}$.  By our inductive assumption, there is some $\slz$-finite collection $\mathcal{F}_k$ of proper $\uh$-special subvarieties of $\uh^{n-k}$.  Every $j'$-generic $\uh$-special point in $J^{-1}(V_k)$ is contained in some $F\in\mathcal{F}_k$.  Let
\[\mathcal{G}_2^*=\{\omega(\uh^k\times F):F\in\mathcal{F}_k,1\leq k< n,\omega\text{ a permutation of coordinates}\}.\]
Then $\mathcal{G}_2^*$ is $\slz$-finite.

Let 
\[\mathcal{G}^*=\mathcal{G}_1^*\cup\mathcal{G}_2^*,\]
and let 
\[V^*=\JClose{\bigcup \mathcal{G}^*}.\]
Since $\mathcal{G}^*$ consists of $\slz$-finitely many proper $\uh$-special subvarieties, $V^*$ is a proper subvariety of $\mathbb{C}^{3n}$.  

Suppose now that $J^{-1}(V\setminus V^*)$ contains $\slz$-infinitely many $j'$-generic quadratic points.  In particular, there is some $j'$-generic quadratic $\tau\in V\setminus V^*$ with $D(\tau)>D=D(V)$.  By \ref{thrm:MainPointCountingThrm}, there is some $\uh$-special set $H$ with
\[\tau\in H \hookrightarrow V.\]
Now, $H$ is a translate of some basic $\uh$-special variety $B\suq\uh^k$.  If $B$ is a proper subvariety of $\uh^k$, then $H$ should have been accounted for by $\mathcal{G}_1^*$, whence $J(H)\suq V^*$, which contradicts $\tau\not\in V^*$.  So we must have $B=\uh^k$.  

So, up to permutation of coordinates, we have
\[H=\uh^k\times\{\tau_{k+1},\dots,\tau_n\}.\]
As in the case $n=1$ above, the fact that $H \hookrightarrow V$ (via some data including some positive real numbers $c_i$) tells us that
\[J(\sigma_1,\dots,\sigma_k,\tau_{k+1},\dots,\tau_n)\in V\]
whenever $\Imm\sigma_i=c_i$.  By analytic continuation we then see that
\[J(\uh^k\times\{\tau_{k+1},\dots,\tau_n\})\suq V,\]
whence
\[\mathbb{C}^{3k}\times\{J(\tau_{k+1},\dots,\tau_n)\}\suq V.\]
So $(\tau_{k+1},\dots,\tau_n)$ is an $\uh$-special point of $J^{-1}(V_k)$.  Moreover, it is $j'$-generic since $\tau$ was.  Hence $(\tau_{k+1},\dots,\tau_n)$ should have been accounted for by $\mathcal{G}_2^*$.  Hence $\tau\in V^*$, which is a contradiction.

So $J^{-1}(V\setminus V^*)$ can only contain $\slz$-finitely many $j'$-generic quadratic points, whence the $j'$-generic quadratic points in $J^{-1}(V)$ are accounted for by the $\slz$-finite collection $\mathcal{G}^*$, together with $\slz$-finitely many additional points.
 \end{proof}

\end{thrmREP}
In the next section, we will state a more precise version of Conjecture \ref{conj:PartialAOwDerivs}, and prove it under the assumption of \ref{conj:AlgIndep}.  Before we can do so, we will take the time now to note the following.  Proposition \ref{propn:BasicSpecialCount} clearly holds uniformly, and by using the uniform Pila-Wilkie Theorem, we can also get a uniform version of Theorem \ref{thrm:MainPointCountingThrm}.  Using this uniformity, it is easy to get the following uniform version of \ref{thrm:PartialAOwDerivs}.

\begin{thrm}\label{thrm:UniformPartialAO}
 Let $V\suq\mathbb{C}^{3n+k}$ be an algebraic variety defined over $\overline{\mathbb{Q}}$, considered as an algebraic family of varieties, \[V_{\mathbf{a}}\suq\mathbb{C}^{3n}, \quad\mathbf{a}\in\mathbb{C}^k.\]
 For each positive integer $r$, there is an $\slz$-finite collection $\sigma_r(V)$, consisting of proper $\uh$-special subvarieties of $\uh^n$, with the following property.  
 
 Whenever $\mathbf{a}\in\overline{\mathbb{Q}}^k$ satisfies $\max [\mathbb{Q}(a_i):\mathbb{Q}]\leq r$ and $V_{\mathbf{a}}$ is a \emph{proper} subvariety of $\mathbb{C}^{3n}$, every $j'$-generic quadratic point in $J^{-1}(V_{\mathbf{a}})$ is contained in some $G\in\sigma_r(V)$.
\end{thrm}

\section{A More Precise Statement}
Readers familiar with the normal shape of Andr\'e-Oort statements may have noticed that \ref{conj:PartialAOwDerivs} is rather weaker than one might expect.  Even taking into account the action of $\slz$, a more natural analogue of the classical case might look like the following.

\begin{statement}\label{conj:FalseAO}
 Let $V\suq\mathbb{C}^{3n}$ be an algebraic variety defined over $\overline{\mathbb{Q}}$.  Then the collection of maximal $J$-special subvarieties of $V$ is $\slz$-finite.
\end{statement}
This turns out to be false.  The reason for its failure is fairly simple; the modular relations that relate $j'(g\tau)$ (for some $g\in\gl$) with $j'(\tau)$ include not just $j(\tau)$, $j'(\tau)$, $j(g\tau)$ and $j'(g\tau)$, but also include instances of $m_g(\tau)=(c\tau+d)^2/N$.  With the right polynomial, one can therefore enforce arbitrary relations between $m_g(\tau)$ and other $m_h(\sigma)$ arising in other coordinates.

Similarly, the modular relation for $j''$ introduces new variables to the equation.  Indeed, by differentiating the modular polynomials one can find a rational function $\mu_N$ (in 7 variables) such that
\[j''(g\tau)=\mu_N(j(\tau),j(g\tau),j'(\tau),j'(g\tau),j''(\tau),c,(c\tau+d)),\]
where $g=\begin{pmatrix}a&b\\c&d\end{pmatrix}$ is a primitive integer matrix of determinant $N$.  Moreover, $\mu_N$ is linear in the $c$-coordinate.  Hence we are able to enforce relations on the $c$ that can arise, as well as the $m_g(\tau)$.  

\begin{eg}Let us see some examples to illustrate this issue.

 \begin{enumerate}
  \item Let $W\suq\mathbb{C}^2$ be an algebraic variety defined over $\mathbb{Q}$.  Suppose that $W$ has at least one solution $(x,y)$ where $x$ and $y$ are both squares of quadratic points in $\uh$.  Fix two positive integers $M$ and $N$. 
 
 Writing a general element of $\mathbb{C}^{12}$ as $(X_1,Y_1,Z_1,\dots,X_4,Y_4,Z_4)$, consider the variety $V\suq\mathbb{C}^{12}$ defined (over $\overline{\mathbb{Q}}$) by
 \[\Phi_M(X_1,X_2)=0, \quad \Phi_N(X_3,X_4)=0,\]
 \[\left(\dfrac{-Y_2}{Y_1\lambda_M(X_1,X_2)},\dfrac{-Y_4}{Y_3\lambda_N(X_3,X_4)}\right)\in W.\]
 Then the special points of $J^{-1}(V)$ are precisely the points $(\tau,g\tau,\sigma,h\sigma)$, where $g$ and $h$ are (arbitrary) primitive integer matrices with determinant $M$ and $N$ respectively, and $\tau$, $\sigma$ are quadratic points satisfying
 \[(m_g(\tau),m_h(\sigma))\in W.\]
 Since $W$ has at least one solution which is a square of a quadratic point, we can certainly find $\tau$ and $\sigma$ to solve this equation.  Indeed, by modifying $\tau$ and $\sigma$ we can solve this equation for any $g$ and $h$ of the right determinant.  The resulting collection of special points is certainly $\slz$-infinite, but no positive-dimensional $\uh$-special variety is contained in $J^{-1}(V)$.  
 
 This example therefore serves as a counterexample to the hypothetical statement \ref{conj:FalseAO}.  Since the points all lie within the $\slz$-translates of the $\uh$-special set
 \[\{(\tau_1,g\tau_1,\tau_2,h\tau_2):\tau_i\in\uh\},\]
 our main theorem \ref{thrm:PartialAOwDerivs} is still fine!  Let us see one more example.
 
  \item Fix a quadratic point $\sigma\in\uh$ and $\gamma=\begin{pmatrix}a&b\\c&d\end{pmatrix}\in\slz$.  Then
  \[(c\sigma+d)^2=m_\gamma(\sigma)=\dfrac{j'(\gamma\sigma)}{j'(\sigma)}\]
  and 
  \[c=\dfrac{j''(\gamma\sigma)-j''(\sigma)(c\sigma+d)^4}{2(c\sigma+d)^3j'(\tau)},\]
  whence 
  \begin{align*}c^2&=\dfrac{\big(j''(\gamma\sigma)-j''(\sigma)(c\sigma+d)^4\big)^2}{4(c\sigma+d)^6j'(\sigma)^2}\\
  &=\dfrac{\left(j''(\gamma\sigma)-j''(\sigma)\left(\dfrac{j'(\gamma\sigma)}{j'(\sigma)}\right)^2\right)^2}{4\left(\dfrac{j'(\gamma\sigma)}{j'(\sigma)}\right)^3j'(\sigma)^2}.\end{align*}
  So for the appropriate rational function $q$ we have
  \[c^2=q(j'(\sigma),j'(\gamma\sigma), j''(\sigma),j''(\gamma\sigma)).\]
  Given a variety $W\suq\mathbb{C}^2$, defined over $\overline{\mathbb{Q}}$, we can then define $V\suq\mathbb{C}^9$ by
  \[\Phi_N(X_1,X_2), \quad X_3=j(\sigma),\]
  \[\forall w\in\mathbb{C}, \left(\dfrac{-Y_2}{Y_1\lambda_N(X_1,X_2)},q\big[w,Y_3, p_{\Imm\sigma}(j(\sigma),\chi^*(\sigma),w),Z_3\big]\right)\in W.\]
  Then the $\uh$-special points of $J^{-1}(V)$ are exactly those points
  \[(\tau,g\tau,\gamma\sigma),\]
  where $g=\begin{pmatrix}a&b\\c&d\end{pmatrix}$ is any primitive integer matrix of determinant $N$, $\gamma=\begin{pmatrix}A&B\\C&D\end{pmatrix}\in\slz$ and 
  \[((c\tau+d)^2,C^2)\in W.\]
  Once again, if $W$ is suitable then this is an $\slz$-infinite collection, but no positive dimensional $\uh$-special set lies in $J^{-1}(V)$.  
 \end{enumerate}
\end{eg}

With variants of the examples above, one can produce varieties whose special points satisfy almost any arbitrary relation, provided the relation is written in terms of variables $c$, $(c\tau+d)$ corresponding to matrices $g\in\gl$, and $C,D$ corresponding to some $\slz$-translate of a fixed $\sigma$.

The idea of our stronger result is that these relations should be the only obstruction to a result like \ref{conj:FalseAO}.  In order to state this precisely, we will need to go through some technicalities.\\

Given a proper $\uh$-special set $G\suq\uh^n$, there is an underlying partition of $\{1,\dots,n\}$ which can be written as
\[S_0\cup S_1\cup\dots\cup S_h\cup T_1\cup\dots\cup T_k,\]
with only $S_0$ allowed to be empty and only the $T_i$ allowed to be singletons.  (The condition that $G$ is proper is equivalent to requiring that $k < n$.)  For $i>0$, let $r_i=\# S_i-1$ and let $s_i$ be the smallest element of $s_i$.  Also associated to $G$ are some matrices $g_{i,1},\dots,g_{i,r_i}\in\gl$, so that each coordinate in $S_i$ (except the $s_i$ coordinate) is defined by $\tau=g_{i,j}\tau_{s_i}$.

Given such a $G$ and given a tuple of matrices
\begin{equation}\label{eqn:SLZComp}\gamma=(\gamma_1,\dots,\gamma_{\#S_0},\gamma_{1,1},\dots,\gamma_{1,r_1},\dots,\gamma_{h,1},\dots,\gamma_{h,r_h})\in\slz^{\#S_0+\sum r_i},\end{equation}
let $c_i,d_i$ be the bottom row of $\gamma_i$, and $c_{i,j},d_{i,j}$ be the bottom row of the matrix $\gamma_{i,j}g_{i,j}$.\\

This is all building towards the following definitions.  

\begin{defn}
 A variety $W\suq \mathbb{C}^{2\# S_0+2\sum r_i}$, defined over $\overline{\mathbb{Q}}$, is called a $G$-variety.  For a $G$-variety $W$ and a given $\gamma\in\slz^{\#S_0+\sum r_i}$ (as above), we define $W^\gamma\suq\uh^{\{s_1,\dots,s_h\}}$ to be the set of $(\tau_{s_1},\dots,\tau_{s_h})$ such that
\[(\dots,c_i,d_i,\dots,c_{i,j}, c_{i,j}\tau_{s_i}+d_{i,j},\dots)\in W.\]
\end{defn}
If $\gamma$ is an element of the full group $\slz^n$, then it consists of $\gamma'\in\slz^{\#S_0+\sum r_i}$, as above, together with some more matrices
\[\alpha_{s_1},\dots,\alpha_{s_h}\in\slz,\]
corresponding to the $s_i$-coordinates, and
\[\beta_1,\dots\beta_k\in\slz\]
corresponding to the singleton coordinates in the $T_i$.  For such a $\gamma\in\slz^n$, we will abuse notation and write
\[W^\gamma=(\alpha_{s_1},\dots,\alpha_{s_h})\cdot W^{\gamma'}.\]
(The $\beta_i$ have no meaningful effect.)  

\begin{notn}
 We will write
 \[\Sp(\gamma,W)=\{(\tau_1,\dots,\tau_n)\in\uh^n: (\tau_{s_1},\dots,\tau_{s_h})\in W^\gamma\text{ and every }\tau_{i}\text{ is quadratic}\}.\]
\end{notn}

In the case where $h=0$, so that we have no $\tau$-coordinates to work with, the variety $W^\gamma$ only enforces conditions on the $\gamma$ corresponding to the $S_0$-coordinates.  In this case, we will use the convention that
\[\Sp(\gamma,W)=\uh^n\]
if $(\dots,c_i,d_i,\dots)\in W$, and
\[\Sp(\gamma,W)=\emptyset\]
otherwise.

Before we can state our more precise version of \ref{conj:PartialAOwDerivs}, we need one more definition.
\begin{defn}
A pair $(G,W)$, with $G$ a proper $\uh$-special set and $W$ a $G$-variety is said to be \emph{geodesically minimal} if 
\[\bigcup_{\gamma\in\slz^n}\Sp(\gamma,W)\]
is \emph{not} contained in any $\slz$-finite collection of proper $\uh$-special varieties.\\
\end{defn}

\begin{thrm}[Precise Modular Andr\'e-Oort with Derivatives]\label{thrm:PrecAOwDerivs}
Assume Conjecture \ref{conj:AlgIndep}.  Let $V\suq\mathbb{C}^{3n}$ be an algebraic variety defined over $\overline{\mathbb{Q}}$.  Then there is a finite collection $\sigma(V)$ of $\uh$-special subvarieties of $\uh^n$, and for each $G\in\sigma(V)$ an associated $G$-variety $W_G$, with the following properties.
\begin{itemize}
 \item For every $G\in\sigma(V)$, $(G,W_G)$ is geodesically minimal.
 \item The set of quadratic points in $J^{-1}(V)$ is precisely
 \[\bigcup_{\substack{G\in\sigma(V)\\\gamma\in\slz^n}}\gamma\cdot G\cap \Sp(\gamma,W).\]
\end{itemize}
\end{thrm}
\begin{proof}
 Under the assumption of Conjecture \ref{conj:AlgIndep}, Theorem \ref{thrm:PartialAOwDerivs} yields a finite collection $\sigma(V)$ of proper $\uh$-special subvarieties, such that the special points of $J^{-1}(V)$ are contained in
 \[\bigcup_{\substack{G\in\sigma(V)\\\gamma\in\slz^n}}\gamma\cdot G.\]
 Let us look first at a single $G\in\sigma(V)$, and associate some data to it, as in the definitions above.  Associated to $G$ is a partition of $\{1,\dots,n\}$,
 \[S_0\cup S_1\cup\dots\cup S_h\cup T_1\cup\dots\cup T_k,\]
 with $T_i$ singletons, $\#S_i>1$.  As above, we have some associated data
 \[s_i=\min S_i,\qquad r_i=\# S_i-1 \text{ for }i>0,\]
 \[\sigma_1,\dots,\sigma_{\#S_0}\in\uh\text{ quadratic}\]
 and
 \[g_{i,j}\in\gl\text{, a primitive integer matrix with determinant }N_{i,j},\text{ for }1\leq i \leq h, 1\leq j\leq r_i.\]
 For ease of notation, we will assume that the coordinates are ordered nicely, with the first few coordinates in $S_0$, the next few in $S_1$, and so on.
 
 Recall: for each $N$, there is a rational function $\mu_N$ with the property that
 \[j''(g\tau)=\mu_N(j(\tau),j(g\tau),j'(\tau),j'(g\tau),j''(\tau),c,(c\tau+d))\]
 whenever $g=\begin{pmatrix}a&b\\c&d\end{pmatrix}$ is a primitive integer matrix of determinant $N$.  Moreover, it will be useful later to know that
 \begin{multline}\label{eqn:muNCalc}\mu_N\left(j(\tau),j(g\tau),j'(\tau),j'(g\tau)\dfrac{(c_0\tau_0+d_0)^2}{(c\tau+d)^2},j''(\tau),c_0,(c_0\tau_0+d_0)\right)\\=j''(g\tau)\dfrac{(c_0\tau_0+d_0)^4}{(c\tau+d)^4}+j'(g\tau)\dfrac{2(c_0\tau_0+d_0)^3(c_0d-cd_0)}{(c\tau+d)^3}.\end{multline}
 This follows by a straightforward, if tedious, calculation.
 
 Given $\gamma\in\slr^{n}$, and a point $\tau=(\tau_1,\dots,\tau_{h})\in\uh^{h}$, define a variety
 \[V_{\gamma,\tau}\suq \mathbb{C}^{3(h+k)}\]
 as follows.  First, write $\gamma$ as in (\ref{eqn:SLZComp}).  That is, $\gamma$ consists of 
 \[\gamma'=(\gamma_1,\dots,\gamma_{\#S_0},\gamma_{1,1},\dots,\gamma_{1,s_1},\dots,\gamma_{h,1},\dots,\gamma_{h,s_h})\in\slr^{\# S_0+\sum r_i},\]
 \[\alpha_{s_1},\dots,\alpha_{s_h}\in\slz\] corresponding to the $s_i$ coordinates, and 
 \[\beta_1,\dots,\beta_k\in\slz\]
 corresponding to the singleton coordinates in the $T_i$.  Let $c_i$, $d_i$ be the bottom row of $\gamma_i$ and $c_{i,j},d_{i,j}$ the bottom row of $\gamma_{i,j}g_{i,j}$.

 Define $V_{\gamma,\tau}'$ by
 \[(X_1,Y_1,Z_1,\dots,X_{h+k},Y_{h+k},Z_{h+k})\in V_{\gamma,\tau}'\]\[\Longleftrightarrow\]
 \begin{multline*}\forall w_{i,j}\in\mathbb{C}\text{ with }\Phi_{N_{i,j}}(X_i,w_{i,j}), i\leq h,j\leq r_i,\\ \Bigl[\dots,j(\sigma_i),j'(\gamma_i\sigma_i),j''(\gamma_i\sigma_i),\dots,X_i,Y_i,Z_i,\dots,w_{i,j},-Y_i(c_{i,j}\tau_{s_i}+d_{i,j})^2\lambda_{N_{i,j}}(X_i,w_{i,j}),\\\mu_{N_{i,j}}\bigl(X_i,w_{i,j},Y_i,-Y_i(c_{i,j}\tau_{s_i}+d_{i,j})^2\lambda_{N_{i,j}}(X_i,w_{i,j}),Z_i,c_{i,j},c_{i,j}\tau_{s_i}+d_{i,j}\bigr),\dots\\\dots,X_{h+i},Y_{h+i},Z_{h+i},\dots\Bigr]\in V.\end{multline*}
 
 Taking $\overline{\mathbb{Q}}$-Zariski closures replaces the $j'(\gamma_i\sigma_i)$ and $j''(\gamma_i\sigma_i)$ by suitable rational functions involving $\sigma_i$, $\Imm\sigma_i$, $\chi^*(\sigma_i)$, $c_i$, $d_i$, and some complex numbers $w$ which are allowed to be arbitrary.  Making these replacements we get a variety $V_{\gamma,\tau}$, defined over $\overline{\mathbb{Q}}$, depending polynomially on $c_i,d_i,c_{i,j},(c_{i,j}\tau_{i}+d_{i,j})$ (and nothing else).  Thus each $V_{\gamma,\tau}$ is a fibre of an algebraic family of varieties $\widehat{V}$, defined over $\overline{\mathbb{Q}}$.  
 
 We now apply Theorem \ref{thrm:UniformPartialAO} to $\widehat{V}$.  We get an $\slz$-finite collection $\sigma_2\bigl(\widehat{V}\bigr)$, consisting of $\uh$-special subvarieties of $\uh^{(h+k)}$, such that for all $\gamma\in\slz^n$ and all quadratic $\tau\in\uh^h$, either \[V_{\gamma,\tau}=\mathbb{C}^{3(h+k)}\] or the $J$-special subvarieties of $V_{\gamma,\tau}$ are accounted for by $\sigma_2\bigl(\widehat{V}\bigr)$.  The $H\in\sigma_2\bigl(\widehat{V}\bigr)$ correspond in the obvious way to an $\slz$-finite collection $\mathcal{G}$ of proper $\uh$-special subvarieties of $G$.  We add all these $H\in\mathcal{G}$ to the overarching collection $\sigma(V)$.\\
 
 Now, a quadratic point lying in \[\bigcup_{\gamma\in\slz^n}\gamma\cdot G\cap J^{-1}(V)\] corresponds to a quadratic point $\tau=(\tau_1,\dots,\tau_n)\in G$ together with $\gamma\in\slz^n$ such that $\gamma\tau\in J^{-1}(V)$.  Such a pair $(\tau, \gamma)$ necessarily satisfies
 \[J(\tau')\in V_{\gamma,\tau},\]
 where
 \[\tau'=(\alpha_{s_1}\tau_{s_1},\dots,\alpha_{s_h}\tau_{s_h},\beta_1\tau_{n-k+1},\dots,\beta_k\tau_n).\]
 By the properties of $\sigma_2\bigl(\widehat{V}\bigr)$, either $\tau'\in H$, for some $H\in\sigma_2\bigl(\widehat{V}\bigr)$, or $V_{\gamma,\tau}=\mathbb{C}^{3(h+k)}$.  In the first case, we have $\tau\in H$, for some $H\in\mathcal{G}$.  
 
 Define a set
 \[R=\left\{(\tau,\gamma)\in\uh^n\times\slz^n:\begin{matrix}\tau\text{ is quadratic}, \tau\in G, J(\gamma\tau)\in V, \\\\\forall H\in\mathcal{G},\tau\not\in H\end{matrix}\right\}.\]
 If $R$ is empty, we can stop here, removing $G$ from $\sigma(V)$ entirely; it contributes no special points other than those already accounted for by $\mathcal{G}$.  
 
 If $R\ne\emptyset$, we continue.  By the properties of $\mathcal{G}$, every $(\tau,\gamma)\in R$ must satisfy
 \[V_{\gamma,\tau}=\mathbb{C}^{3(h+k)}.\]
  Hence, in the definition of $V_{\gamma,\tau}$, we can replace $(X_1,\dots,Z_{h+k})$ with $J(z_1,\dots,z_{h+k})$, for \emph{arbitrary} $z_i\in\uh$.  By (\ref{eqn:muNCalc}) and an easy calculation involving $\lambda_N$, we see that
 \begin{multline*}\label{eqn:DefnOfWz}\Bigl[\dots,J(\gamma_i\sigma_i),\dots,J(z_i),\dots\\\dots,j(h_{i,j}z_i),j'(h_{i,j}z_i)\dfrac{(c_{i,j}\tau_{s_i}+d_{i,j})^2}{\delta_{i,j}^2},j''(h_{i,j}z_i)\dfrac{(c_{i,j}\tau_{s_i}+d_{i,j})^4}{\delta_{i,j}^4}+2j'(h_{i,j}z_i)\dfrac{c_{i,j}(c_{i,j}\tau_{s_i}+d_{i,j})^3}{\delta_{i,j}^2},\dots\\\dots,J(z_{h+i}),\dots\Bigr]\in V\end{multline*}
 for all $(z_1,\dots,z_{h+k})\in\uh^{h+k}$.  Here $h_{i,j}$ is an upper triangular matrix in the $\slz$-orbit of $g_{i,j}$ and $\delta_{i,j}$ is its lower-right entry. 
 
 As before, by taking $\overline{\mathbb{Q}}$-Zariski closures we can replace the $J(\gamma_i\sigma_i)$ by some $\overline{\mathbb{Q}}$-rational functions of $c_i$ and $d_i$.  Thus the above equation, for each choice of $\mathbf{z}=(z_1,\dots,z_{h+k})$, defines a $G$-variety which we will call $W_{\mathbf{z}}$.  The intersection 
 \[W=\bigcap_{\mathbf{z}\in\uh^{h+k}}W_{\mathbf{z}}\]
 is still a $G$-variety.  (It is nonempty since $R$ is nonempty.)
 
 We have seen that every $(\tau,\gamma)\in R$ satisfies 
 \[(\tau_{s_1},\dots,\tau_{s_h})\in W^\gamma.\]
 Conversely, any quadratic $\mathbf{z}=(z_1,\dots,z_h)$, no matter its height, which is a solution of some $W^\gamma$, must come from a member of $\gamma\cdot G\cap J^{-1}(V)$.  Indeed, in this situation we have \begin{multline*}\Bigl[\dots,J(\gamma_i\sigma_i),\dots,J(z_i),\dots\\\dots,j(h_{i,j}z_i),j'(h_{i,j}z_i)\dfrac{(c_{i,j}z_i+d_{i,j})^2}{\delta_{i,j}^2},j''(h_{i,j}z_i)\dfrac{(c_{i,j}z_i+d_{i,j})^4}{\delta_{i,j}^4}+2j'(h_{i,j}z_i)\dfrac{c_{i,j}(c_{i,j}z_i+d_{i,j})^3}{\delta_{i,j}^2},\dots\\\dots,J(z_{h+i}),\dots\Bigr]\in V,\end{multline*}
 for all $(z_{h+1},\dots,z_{h+k})\in\uh^k$.
 
 Brief calculations involving the transformation laws for $j'$ and $j''$ show that 
 \[j'(h_{i,j}z_i)\dfrac{(c_{i,j}z_i+d_{i,j})^2}{\delta_{i,j}^2}=j'(\gamma_{i,j}g_{i,j}z_i)\dfrac{(c_{i,j}z_i+d_{i,j})^2}{(c_{i,j}z_i+d_{i,j})^2}=j'(\gamma_{i,j}g_{i,j}z_i)\]
 and that
 \[j''(h_{i,j}z_i)\dfrac{(c_{i,j}z_i+d_{i,j})^4}{\delta_{i,j}^4}+2j'(h_{i,j}z_i)\dfrac{c_{i,j}(c_{i,j}z_i+d_{i,j})^3}{\delta_{i,j}^2}=j''(\gamma_{i,j}g_{i,j}z_i)\dfrac{(c_{i,j}z_i+d_{i,j})^4}{(c_{i,j}z_i+d_{i,j})^4}=j''(\gamma_{i,j}g_{i,j}z_i).\]
 Thus we get
 \begin{equation*}\Bigl[\dots,J(\gamma_i\sigma_i),\dots,J(z_i),\dots,j(\gamma_{i,j}g_{i,j}z_i),j'(\gamma_{i,j}g_{i,j}z_i),j''(\gamma_{i,j}g_{i,j}z_i),\dots,J(z_{h+i}),\dots\Bigr]\in V,\end{equation*}
 for every $(z_{h+1},\dots,z_{h+k})\in\uh^k$.  Hence $\mathbf{z}$ corresponds to the point
 \[(\dots,\gamma_i\sigma_i,\dots,z_i,\dots,\gamma_ig_{i,j}z_i,\dots,z_{h+i},\dots)\in \gamma\cdot G\cap J^{-1}(V),\]
 for any choice of $z_{h+i}$.\\
 
 To sum up, we have seen that the quadratic points in $\gamma\cdot G\cap J^{-1}(V)$ consist precisely of:
 \begin{itemize}
  \item Those quadratic points that lie in some $H\in\mathcal{G}$.  
  \item Those points that corresponds to quadratic solutions of $W^\gamma$, that is
  \[\gamma\cdot G\cap \Sp(\gamma, W).\]
 \end{itemize}
 We are not claiming that these possibilities are mutually exclusive!
 
 Before we are done with $G$, we must check whether $(G,W)$ is geodesically minimal.  If it does happen to be geodesically minimal, we are done.  Otherwise, by definition, 
 \[\bigcup_{\gamma\in\slz^n}\Sp(\gamma, W)\]
 is contained in some $\slz$-finite collection of proper $\uh$-special subvarieties of $\uh^{h+k}$.  This is not a problem; we simply remove $G$ from $\sigma(V)$ entirely and replace it by the appropriate finite collection of proper subvarieties of $G$.\\
 
 This is as much as we can do with a given $G\in\sigma(V)$.  It may be helpful to have a brief summary here.  For the given $G$ we have done two things:
 \begin{enumerate}
  \item Either removed $G$ from $\sigma(V)$ entirely or associated to $G$ an geodesically minimal $G$-variety $W_G$.
  \item Added to $\sigma(V)$ some finite collection $\mathcal{G}$ of proper subvarieties of $G$.
 \end{enumerate}
 Moreover, the union of the $\uh$-special subvarieties of $\gamma\cdot G\cap J^{-1}(V)$ is precisely
 \[\gamma\cdot G\cap \Sp(\gamma,W_G),\]
 together with the $\uh$-special subvarieties of
 \[\bigcup_{H\in\mathcal{G}}\gamma\cdot H\cap J^{-1}(V).\] 
 
 This is enough to conclude the theorem.  Simply perform the above process to each $G\in\sigma(V)$ in turn, taking the $G$ in \emph{descending order of dimension}.  Since each $G$ can add to $\sigma(V)$ only finitely many varieties of strictly smaller dimension, the process will eventually terminate.
\end{proof}

\section{Uniformity}
In this final section, I will discuss uniform versions of the two main results of the document, Theorems \ref{thrm:PartialAOwDerivs} and \ref{thrm:PrecAOwDerivs}.  For this, we are closely following work of Scanlon, who gives in \cite{Scanlon2004} a very general approach to uniformising results of this type.  Unfortunately, our setting does not fit perfectly into Scanlon's framework; the full strength of his result is therefore not available to us.  For our purposes it is enough that the central ideas in Scanlon's work \emph{do} apply.  

There are two main lemmas of Scanlon's that we will use.  The first is Lemma 3.1 from \cite{Scanlon2004}, which applies directly.  We write it out here for completeness.

\begin{lma}\label{lma:PassingUpFields}
 Let $k$ be a field and $K$ an algebraically closed field extension of $k$.  Let $X$ be a variety over $k$ and $X_K$ its base change to $K$.  Let $A\suq X(K)$.  Suppose that $Y\suq X_K$ is constructible.  Then there is a natural number $n$, some constructible set $Z\suq X\times X^n$, defined over $k$, and some $a\in A^n$ such that $Z_a(K)\cap A=Y(K)\cap A$.
\begin{proof}
 See Lemma 3.1 from \cite{Scanlon2004}
\end{proof}

 \end{lma}
 
 The next lemma is the analogue of Lemma 3.2 from \cite{Scanlon2004}.  The statement needs some very slight modification before it applies in our setting, but the proof of the modified lemma is essentially identical to the original.

\begin{lma}\label{lma:UniformConstructible}
 Let $k\suq K$ be algebraically closed fields, $X=\mathbb{A}^s$, $B=\mathbb{A}^t$ and $Y\suq X\times B$ a constructible subset.  Let $A\suq X(K)$.  Suppose that there exist $\alpha,\beta\in A$ with the following property:
 
 For $m\in \mathbb{N}$ and $i\leq m$, let $p^{(m)}_i$ be the $m$-tuple
 \[(\alpha,\dots,\alpha,\beta,\alpha,\dots,\alpha),\]
 with the $\beta$ arising in the $i$th place.  Let $P^{(m)}_i$ be the $k$-Zariski closure of $\left\{p^{(m)}_i\right\}$.  Then 
 \[p^{(m)}_i\not\in P^{(m)}_j\]
 for any $i\ne j$.
 
 Then there is a natural number $n$ and a $k$-constructible set $Z\suq X\times X^n$ such that for any $b\in B(K)$, there is $a\in A^n$ for which $Y_b(K)\cap A=Z_a\cap A$.
 \begin{proof}
  Throughout this proof we will suppress notation, writing $W=W(K)$ whenever $W$ is a constructible set over $K$.
  
  Consider $K$ as a structure in the language $\mathcal{L}=(+,\cdot,P_A,\{a\}_{a\in K})$, where each constant symbol $a$ is to be interpreted as the corresponding element $a\in K$ and $P_A$ is an $s$-ary predicate to be interpreted as the set $A$.  In this language, we can express the condition ``$x\in A\cap W$'' for any affine variety $W$ over $K$.  
  
  Let $T$ be the $\mathcal{L}$-theory of $K$, and then let $\mathcal{L}'$ be $\mathcal{L}$ together with some new constant symbols $b_1,\dots, b_t$.  Write $b=(b_1,\dots,b_t)$.
  
  Let $\mathcal{C}(k)$ be the set of $k$-constructible subsets of $X\times X^N$ (for some $N$) and consider the set
  \[\Gamma = T\cup \{\forall c\in A^N\quad \exists x\in A\quad (x\in Y_b\setminus Z_c\vee x\in Z_c\setminus Y_b): Z\in \mathcal{C}(k)\}.\]
  Suppose that $\Gamma$ is \emph{not} finitely satisfiable.  Let $\Gamma_0$ be a finite subset witnessing this.  Since $T$ is the theory of $K$, $\Gamma_0$ cannot be contained in $T$, so it mentions some finitely many $k$-constructible sets $Z_1,\dots,Z_l$, with $Z_i\suq X\times X^{N_i}$.  Since  $\Gamma_0$ is not satisfiable, we have:
  \[\forall b\in B \quad\exists i\leq l\quad \exists c\in A^{N_i}\quad \forall x\in A\quad( x\not\in Y_b\setminus Z_c\wedge x\not\in (Z_i)_c\setminus Y_b).\]
  In other words, for every $b\in B$, there is some $Z_i$ and some $c\in A^{N_i}$ such that
  \[A\cap (Z_i)_c=A\cap Y_b.\]
  Now let $N=\max\{N_i:i\leq l\}$ and let $n=N+l$.  Define 
  \[Z=\bigcup_{i=1}^{l}Z_i\times X^{N-N_i}\times P^{(l)}_i,\]
  with $P^{(l)}_i$ as in the hypotheses of the lemma.  This $Z$ is a constructible set defined over $k$, and satisfies the conclusion of the lemma: if $b\in B$, then for some $i$, $c$, we have $A\cap (Z_i)_c=A\cap Y_b$. Letting $c'=(c,\alpha,\dots,\alpha, p^{(l)}_i)\in A^{N+l}$, we get
  \[A\cap Z_{c'}=A\cap (Z_i)_c=A\cap Y_b.\]
  So suppose on the other hand that $\Gamma$ is finitely satisfiable.  By the Compactness Theorem, it is satisfiable.  So we have an algebraically closed extension $L\supseteq K$ of $K$, a point $b\in B(L)$ and a set $A^*\suq X(L)$ such that, for every $k$-constructible $Z\suq X\times X^N$ and $c\in (A^*)^N$, we have
  \[Y_b(L)\cap A^*\ne Z_c(L)\cap A^*.\]
  This contradicts Lemma \ref{lma:PassingUpFields} (applied to $k$, $L$, $A^*$, $X$ and $Y_b$).
 \end{proof}

\end{lma}
Using this, we would like to get a uniform version of \ref{thrm:PrecAOwDerivs}. This can indeed be done, though perhaps not in exactly the manner we might like.  

Given an algebraic family $V\suq \mathbb{C}^{3n+k}$, we can apply \ref{lma:UniformConstructible}, with 
\[A=\{J(\tau):\tau\in\uh^n \text{ quadratic}\}.\]
For the resulting constructible set $Z$, we can write
\[Z=\bigcup_{i=1}^{r}X_i\setminus Y_i\]
for some varieties $X_i, Y_i$ defined over $\overline{\mathbb{Q}}$.   If we apply \ref{thrm:PrecAOwDerivs} (under the assumption of Conjecture \ref{conj:AlgIndep}) to each of the $X_i$ and $Y_i$, we get a finite collection $\sigma(V)$, consisting of \emph{pairs} $(G,H)$ of $\uh$-special varieties, with corresponding $G$-varieties $(W_G,W_H)$, such that the union of the $\uh$-special subvarieties of $J^{-1}(V_b)$, for any fibre $b$, is precisely the fibre of the set
\[\bigcup_{(G,H)\in\sigma(V)}\left[\bigcup_{\gamma\in\slz} \gamma G\cap \operatorname{Sp}(\gamma,W_G)\mathbin{\Bigg\backslash}\bigcup_{\gamma\in\slz}\gamma H\cap\operatorname{Sp}(\gamma,W_H)\right]\]
at some quadratic $\tau=\tau(b)$.  

This is not quite as good as we might like; one would prefer not to have the $\uh$-special sets $H$ in the picture.  These arise thanks to the fact that $Z$ is only constructible, rather than Zariski closed.  It does not seem possible to get around this; the reason is essentially the same as the reason why the full strategy outlined in Scanlon's paper \cite{Scanlon2004} does not work. 

Let $G$ and $H$ be $\uh$-special varieties and take a corresponding $G$-variety $W_G$ and $H$-variety $W_H$.  If it were the case that $(G\cap H, W_G\cap W_H)$ was geodesically minimal whenever $(G,W_G)$ and $(H,W_H)$ were, then we would be able to apply the rest of Scanlon's work.  This is not necessarily the case, so the rest of Scanlon's work cannot be applied.  Hence the above seems likely to be the best possible uniform version of \ref{thrm:PrecAOwDerivs}.

Theorem \ref{thrm:PartialAOwDerivs}, however, can be uniformised more cleanly, since it is less precise.

\begin{propn}
 Assume Conjecture \ref{conj:AlgIndep}.  Let $V\suq\mathbb{C}^{3n+k}$ be an algebraic variety (with arbitrary field of definition), considered as an algebraic family of fibres $V_b\suq\mathbb{C}^{3n}$.  There is a natural number $N$, and an $\slz$-finite collection $\sigma(V)$ of $\uh$-special subvarieties of $\uh^{n+N}$, with the following property.

For every $b\in\mathbb{C}^k$ with $V_b\ne \mathbb{C}^{3n}$, the $\uh$-special points of $J^{-1}(V_b)$ are contained in
\[\bigcup_{G\in\sigma(V)} G_\tau,\]
where $G_\tau$ is the fibre of $G$ at some fixed quadratic $\tau=\tau(b)\in\uh^N$.
Moreover, all of the $G_\tau$ are \emph{proper} subvarieties of $\uh^n$.
 \begin{proof}
  Let $V\suq\mathbb{C}^{3n+k}$ be a variety, considered as an algebraic family of fibres $V_b\suq\mathbb{C}^{3n}$.  We will apply Lemma \ref{lma:UniformConstructible}, with $X=\mathbb{C}^{3n}$, $Y=V$, $K=\mathbb{C}$, $k=\overline{\mathbb{Q}}$ and
  \[A=\{J(\tau):\tau\in\uh^n \text{ quadratic}\}.\]
  To apply the lemma, we need to find two suitable points $\alpha,\beta\in A$.  This is easy; we only need the $j$-coordinates of $\alpha$ and $\beta$ to be distinct.  So Lemma \ref{lma:UniformConstructible} does apply; we get a $\overline{\mathbb{Q}}$-constructible set $Z\suq \mathbb{C}^{3(n+dn)}$ such that, for every $b\in\mathbb{C}^k$, there is some $a\in A^{d}$ such that
  \[Z_a\cap A = V_b\cap A.\]
  Write
  \[Z=\bigcup_{i=1}^{r}X_i\setminus Y_i,\]
  for some $\overline{\mathbb{Q}}$-varieties $X_i$ and $Y_i$.  We will apply Theorem \ref{thrm:PrecAOwDerivs} to each $X_i$ and $Y_i$ separately.  We get some finite sets $\sigma(X_i)$ and $\sigma(Y_i)$ of $\uh$-special varieties, with associated $G$-varieties, exactly describing the special subvarieties of the $Z_i$ in the manner described in the statement of Theorem \ref{thrm:PrecAOwDerivs}.\\
  
  Now, given $b\in\mathbb{C}^{3d}$, let $a\in A^d$ be such that $Z_a\cap A=V_b\cap A$.  Let $\tau$ be a preimage of $a$ under $J$.  
  
  First suppose that no $\sigma(X_i)$ contains any $G$ such that, for some $\gamma\in\slz^N$, $G_{\gamma\tau}=\uh^n$.   Then we are done; the special points of $J^{-1}(V_b)$ are contained in the $\slz$-finite collection of proper $\uh$-special varieties
  \[\{\gamma'\cdot G_{\gamma\tau}:G\in\sigma(X_i), i\leq r,\gamma\in\slz^N,\gamma'\in\slz^n\}.\]
  On the other hand, suppose that some $\sigma(X_i)$ contains $G$ such that, for some $\gamma\in\slz^N$, $G_{\gamma\tau}=\uh^n$.  Then the $G$-variety associated to $G$ cannot impose any condition on the coordinates corresponding to $\uh^n$.  Hence by the properties laid out in \ref{thrm:PrecAOwDerivs}, we must have $(X_i)_a=\mathbb{C}^{3n}$.  
  
  Now apply the same argument to $Y_i$.  There are 2 possibilities.  Either:
  
  \begin{itemize}
   \item The special points of $(Y_i)_a$ are contained in an $\slz$-finite collection 
   \[\{\gamma'\cdot G_{\gamma\tau}:G\in\sigma(Y_i),i\leq r,\gamma\in\slz^N,\gamma'\in\slz^n\},\]
   with each $G_{\gamma\tau}$ being a \emph{proper} $\uh$-special subvariety of $\uh^n$, or
   \item $(Y_i)_a=\mathbb{C}^{3n}$.
  \end{itemize}
  In the first case, since the special points of $(Y_i)_a$ are contained in a lower-dimensional set, it follows that the special points of $Z_a$ are Zariski dense, whence $V_b=\mathbb{C}^{3n}$.  In the second case, $(X_i)_a\setminus (Y_i)_a$ contributes no new special points, so we can ignore this $i$ and move on.
  
  Finally, note that if $G_\tau=\uh^n$ for some $\tau$, then it must be the case for every $\tau'$ that either $G_{\tau'}=\uh^n$ or $G_{\tau'}=\emptyset$.  Thus, if some $G\in\sigma(X_i)$ has the property that (for some $\tau$), $G_\tau=\uh^n$, we can safely remove it.  By the previous arguments, all the special points will still be covered by the rest of the $G\in\bigcup\sigma(X_i)$, except in the case where $V_b=\mathbb{C}^{3n}$.
  
%
%
%
  Hence the ($\slz$-finite) collection
  \[\sigma(V)=\{\gamma\cdot G: \gamma\in\slz^n, G\in\sigma(X_i)\text{ for some }i,\text{ and for every }\tau\in\uh^N, G_\tau\ne\uh^n\}\]
  satisfies the conclusion of the proposition.
 \end{proof}

\end{propn}

To conclude, we will state one final corollary of the above result, which simply says that Theorem \ref{thrm:PartialAOwDerivs} holds for arbitrary varieties, rather than just those defined over $\overline{\mathbb{Q}}$.

\begin{cor}
 Assume Conjecture \ref{conj:AlgIndep}. Let $V\suq\mathbb{C}^{3n}$ be a proper algebraic variety (with arbitrary field of definition).  There exists an $\slz$-finite collection $\sigma(V)$, consisting of \emph{proper} $\uh$-special varieties of $\uh^n$, such that every $\uh$-special point in $J^{-1}(V)$ is contained in some $G\in\sigma(V)$.
 \begin{proof}
  Immediate.
 \end{proof}

\end{cor}

\bibliographystyle{../../bib/scabbrv}
\bibliography{../../bib/thebib}
\end{document}